\documentclass{amsart}
\usepackage{amsmath,amssymb,epic,graphicx,mathrsfs,enumerate,xcolor,amsthm,hyperref,bbold}

\newtheorem{theor}{Theorem}[section]
\newtheorem{defi}[theor]{Definition}
\newtheorem{prop}[theor]{Proposition}

\newtheorem{lemma}[theor]{Lemma}

\newtheorem{corol}[theor]{Corollary}

\newcommand{\GL}{{\operatorname{GL}}}
\newcommand{\SL}{{\operatorname{SL}}}
\newcommand{\SO}{{\operatorname{SO}}}
\newcommand{\PGL}{{\operatorname{PGL}}}

\newcommand{\PSL}{{\operatorname{PSL}}}
\newcommand{\GU}{{\operatorname{GU}}}
\newcommand{\SU}{{\operatorname{SU}}}

\newcommand{\PU}{{\operatorname{PU}}}
\newcommand{\U}{{\operatorname{U}}}
\newcommand{\PSU}{{\operatorname{PSU}}}
\newcommand{\Sp}{{\operatorname{Sp}}}
\newcommand{\PSp}{{\operatorname{PSp}}}
\newcommand{\Sz}{{\operatorname{Sz}}}

\newcommand{\diag}{\mathop{\rm diag}\nolimits}
\newcommand{\Id}{\mathop{\rm Id}\nolimits}
\newcommand{\cn}{\mathop{\rm cn}\nolimits}

\newcommand{\si}{\sigma}

\title[The conjugacy class exponent of simple groups]{On the conjugacy class exponent of the nonabelian simple groups}
\author{}
\date{}

\author{Martino Garonzi}
\address{Martino Garonzi. University of Ferrara (Italy), Dipartimento di Matematica e Informatica.
ORCID: https://orcid.org/0000-0003-0041-3131}
\email{martino.garonzi@unife.it}

\author{Christe Montijo} 
\address{Christe Montijo. Departamento de Matem\'atica, Universidade de Bras\'ilia, Campus 
Universit\'ario Darcy Ribeiro, Bras\'ilia-DF, 70910-900, Brazil.}
\email{christe.montijo@hotmail.com}

\author{Alexandre Zalesski}
\address{Alexandre Zalesski. Departamento de Matem\'atica, Universidade de Bras\'ilia, Campus 
Universit\'ario Darcy Ribeiro, Bras\'ilia-DF, 70910-900, Brazil.
\newline
ORCID: https://orcid.org/0000-0002-9582-9991}
\email{alexandre.zalesski@gmail.com}

\date{}

\subjclass[2020]{20D05,20E32,20E45,20G40}

\keywords{Finite simple groups, conjugacy classes,  products of conjugate elements}

\begin{document}

\begin{abstract}
The generalized order $e_G(g)$ of an element $g$ of a group $G$ is the smallest positive integer $k$ such that there exist $x_1,\ldots,x_k \in G$ such that $g^{x_1} \ldots g^{x_k}=1$, where $g^x=x^{-1}gx$. Let $e(G) = \max \{e_G(g)\ |\ g \in G\}$. We provide upper bounds for $e(G)$  for every finite simple group  $G$. In particular, we show that $e(G)\leq 8$ unless $G\in\{\PSL_n(q), \PSU_n(q), E_6(q),{}^2E_6(q)\}$. For the latter groups $e(G)\leq n,3n+3,36,36$, respectively. In addition, we bound from above the generalized order of semisimple and unipotent elements of finite simple groups of Lie type.\end{abstract}

\maketitle


\section{Introduction}

\subsection{Conjugacy class multiplication and the conjugacy class exponent}

The notion of a conjugacy class is one of the most fundamental ones in group theory.
If $C_1,C_2$ are conjugacy classes of a group $G$ then the product $C_1C_2$ is defined as the set $\{xy: x\in C_1,y\in C_2\}$.  One can similarly defines the product of several conjugacy classes. It is clear that $C_1C_2$ is  a union of conjugacy classes.  One of deep problems recently discussed in literature  is whether $C_1C_2=C_3$ for non-trivial conjugacy classes $C_1,C_2,C_3$ of a finite simple group \cite{AH,ASH,GMT}. In general,
  the problem of determining the conjugacy classes $C$ that are contained in  the product $C_1\cdots C_k$ of given conjugacy  classes $C_1,\ldots ,C_k$ is one of the most complex in finite group theory.
Numerous special cases of this problem were considered, forming a large research area of group theory. The overwhelming number of results obtained concern groups that have an explicit parametrization of the conjugacy classes, such as symmetric groups and simple groups.  In case of $G = \GL_n(F)$, the general linear group over a field $F$, 
this is one of difficult problems of linear algebra; the most significant results are probably obtained in \cite{Cr}, see also \cite{Ku,Lev94,So}.

The above problem can be restated   as that of the existence of   
elements $g_1\in C_1$, \ldots, $g_k\in C_k$ and $g\in C=C_{k+1}$ such that $ g_1\cdots g_{k+1}=1$. In this form the problem in question goes back to Burnside \cite[\S 223]{Br}.
 Explicitly, a formula for  the number $n(C_1,\ldots,C_{k+1})$ of the $(k+1)$-tuples $g_1,\ldots,g_k$ such that $g_1\cdots g_{k+1}=1$ in terms of irreducible characters of $G $, currently well known (see ~\cite[Lemma 10.10]{ASH}, \cite[Theorem 7.2.1]{Se}, \cite[Theorem 2.5.9]{LuP}),  probably first appeared in \cite{Rb}. 
  This is a very useful tool for approaching the problem for groups $G$ for which the irreducible character values are known. Note that we only use the special case with 
$C_1=C_2=\cdots =C_{k+1}$, and mainly for $k=2$, see Lemma \ref{eq} below.

In this paper we address to a special case of the above problem:

\textbf{Problem 1}. Given a non-abelian simple group $G$ and a conjugacy class $C$ of $G$, determine the minimal number $e=e(C)=e_G(C)$ such that $1\in C^e$.

We call $e(C)$ the {\it exponent of} $C$ (in $G$). We set $e(g):=e(C)$ if $g \in C$ and 
 $e(G) := \max_{g \in G} e(g)$. 
We refer to $e(G)$ as a {\it conjugacy class exponent of} $G$. We also discuss 

\textbf{Problem 2}. Determine the conjugacy class exponent of the non-abelian finite simple groups.

Our results reveal that most conjugacy classes of a finite simple group have exponent 2 or 3,
but there are infinitely many such groups $G$ with $e(G)>3$.

\subsection{The conjugacy class exponent of finite simple groups}

If $e(C) \leq 2$ then $g^{-1}\in C$ for all $g \in C$, that is, $g$ is conjugate to its inverse. Such elements are called {\it real}, as well as the classes they belong to.  
It is well-known that an element is real if and only if its complex character values are real numbers. See \cite[Proposition 1(iii)]{FG}, see also \cite[Lemma 2.10]{DZ}  for a more explicit form of this result. The simple groups whose all
conjugacy classes are  real are determined in \cite{TZ05}. This immediately
determines the simple group $G$ with $e(G)=2$; for readers' convenience we record the list here:  

\begin{theor} \label{un4} {\rm \cite[Theorem 1.2]{TZ05}} 
Let $G$ be a finite simple group. All elements in $G$ are real if and only if $G$ is in the  following list:

$(1)$ $\PSp_{2n}(q)$, $n\geq 1$, $q\not\equiv 3\pmod 4;$

$(2)$ $\Omega_{2n+1}(q)$, $n>2$, $q\equiv 1\pmod 4$ or $\Omega_9(q)$, $q\equiv 3\pmod 4;$ 

$(3)$  $P\Omega^+_{4n}(q)$, $n>2$, $q\not\equiv 3\pmod 4;$ 

$(4)$ $P\Omega^-_{4n}(q)$, $n\geq 2;$

$(5)$ ${}^3D_4(q)$ and $P\Omega_8^+(q);$

$(6)$ $A_{10},A_{14}, J_1,J_2.$

\noindent Consequently, $e(G)=2$ if and only if G is one of the above groups, and $e(G)\geq 3$ for all other   groups $G$. 
\end{theor}

In \cite[Conjecture 3.6]{BSS}, R. Bastos, D. Silveira and C. Schneider conjectured that, if $G$ is a finite nonabelian simple group, then $e(G) \leq 3$. This is however not true (see Lemma \ref{lemmaPSL}) and naturally leads to the following problem:

\textbf{Problem 3}. Determine the finite simple groups $G$ such that $e(G)=3$. 

This problem seems to be  very difficult. The partial results we obtain are recorded in the following statement:

\begin{theor}\label{tt2} 
Let $G$ be a finite simple group. Then  $e(G)= 3$ if $G$ occurs in the following list:

$(1)$ Alternating group $A_n$ with $n>4$ and $n\neq 5,6,10,14;$ 

$(2)$ $G$ is a sporadic simple group other than $J_1,J_2;$ 

$(3)$ $\PSL_2(q)$ for $q>3$ and $q \equiv 3 \mod 4;$ 

$(4)$ $\PSp_{4}(q)$,
$q \equiv 3 \mod 4;$ 

$(5)$ $G_2(q)$, q even and ${^2}G_2(q)$ for $q=3^{2m+1}>3;$ 

$(6)$ ${\rm Suz} (q) = {}^2B_2(q)$ and ${}^2F_4(2)' ;$  


$(7)$ $\PSL_3(q)$; 

$(8)$ $\PSU_3(q)$ for $q>2;$ 

$(9)$ $\PSL_4(2), \PSL_5(2)$, $\PSL_4(3)$ and $\PSU_4(3)$. 
\end{theor}

The list in Theorem \ref{tt2} is not expected to be complete, for instance the question
of whether $e(\PSp_6(q))=3$ for  $4|(q+1)$ remains open. 

As mentioned above, there are finite simple groups $G$ with $e(G)>3$; we have computed $e(G)$ with $e(G)>3$ for the following cases:

\begin{theor}\label{tt3} Let $G$ be a finite simple group with $e(G)>3$.

$(1)$ $e(G)=n$ if $G=\PSL_n(q)$, $n>3,$  $(q-1,n)=1$, $q>n$;
 
$(2)$ $e(G)=n$ if $G=\PSU_n(q)$, 
$n=2^k>3$, $q$ even,   $q\geq n$; 

$(3)$ $e(G)=4$ if $G=P\Omega^+ _{2n}(q)$, $n$ odd, $q\neq 2,3,4,5,7,9,13,25;$

$(4)$ $e(G)=4$ if $G=P\Omega^- _{2n}(q)$, $n$ odd, 
and $q\neq 2,3,5,7,11,23$.

\end{theor}

There is a connection between $e(G) $ and the {\it covering number} $\cn_G(C)$ of a conjugacy class $C$ of $G$. Assuming that $G=\langle C\rangle$, this is the minimal $k$ such that $C^k=G$. For $G$ simple, we define $\cn(G)$, the covering number of $G$,  as the maximum of $\cn_G(C)$ over the non-trivial classes $C$ of $G$.
(Following \cite[p. 222]{AH},  we can extend this notion to arbitrary groups but this is not essential for our purposes.)  
Obviously, $e_G(C)\leq \cn_G(C)$ and $e(G) \leq \cn(G)$. In particular, if $\cn(G)=3$, then  $e(G)\leq 3$ and  $e(G)=3$ unless $G$ is as in Theorem \ref{un4} where $e(G)=2$. 
In \cite[Chapters 4,5]{AH} there are examples of simple groups with  $\cn(G)=3$, 
see Proposition \ref{ah8} below.

Lev \cite[Theorem 1]{Lev1996} proved that the covering number of $\PSL_n(F)$ is equal to $n$ if $n > 3$ and $F$ is an arbitrary field  of cardinality at least 4; for  $n=3$ he additionally assumes $F$ to be finite or algebraically closed. So $e(\PSL_n(q)) \leq n$ for all $n \geq 3$ and $q \geq 4$. For $q=2,3$ we prove the following result (see Table \ref{Table_classical}): 

\begin{prop} \label{uuu3} The following holds. 

$(1)$ $e(\PSL_n(2))\leq 6;$ 


$(2)$ $e(\PSL_n(3)) \leq 6$; 

$(3)$ $e(\PSU_n(2)) \leq 18.$ 

\end{prop}

The following problem has attracted a lot of attention.

\textbf{Problem 4}. For a finite simple group $G$ and a non-trivial conjugacy class $C$ determine $\cn_G(C)$.

Note that $e_G(C)-1$ is equal to the minimal integer $l>0$ such that $C^{-1}\subset C^{l}$. In this form   Problems 1,2 are related with a problem on the diameter of a Cayley graph for finite simple group \cite{LScn, LL}; the latter requires computing $\cn_{G}(C\cup C^{-1})$, the minimal number $d$ such that $G=(C\cup C^{-1})^d$.
We clearly have $\cn_G(C)\leq \cn_{G}(C\cup C^{-1}) \cdot (e_G(C)-1)$.

Significant results on computing $\cn(G)$ are obtained in \cite{AH}. Later Lev  \cite{Lev1996} proved that $\cn(\PSL_n(q))=n$ for $q>3$. 
The most advanced results on Problem 4 are obtained in \cite{egh}, where the authors obtained upper bounds for $\cn(G)$ for finite simple groups of Lie type. They showed
that $\cn(G)$ is bounded in terms of the rank of $G$,  and proved a number of results on $\cn_G(C)$ for specific conjugacy classes $C$ of $G$. We benefit from those results as $e(G)\leq \cn(G)$.  As mentioned above, $e(G)=\cn(G)$ for infinitely many groups $\PSL_n(q)$. We wish to specify the following problem.

\textbf{Problem 5}. Determine simple finite groups $G$ with $e(G)=\cn(G)$. 
Given a simple group $G$, determine the conjugacy classes $C$ of $G$ with $e_G(C)=\cn_G(C)$. 

Our experience makes evident that the equality  $e(G)=\cn(G)$ is a rather rare phenomenon. 
As mentioned above, the equality holds if $\cn(G)=3$ and $e(G)> 2$. Using \cite[Tables on p.61,62, Theorems 2.1 and 4.1, pp. 223,239]{AH} we have 

\begin{prop}\label{ah8} Let G be a finite simple group. 

$(1)$ $e(G) = \cn(G)=2$ if and only if  $G\cong   J_1$.

$(2)$ $e(G) = \cn(G)=3$ for the following simple groups:

\noindent $\PSL_2(q), q>3, 4|(q+3)$, $\SL_3(2)$, $\PSL_3(3)$, $\PSL_3(4)$, $\PSL_3(5)$, $\Sp_4(4)$, $\Sz(q),q>2$, $\PSU_3(5)$, $M_{11}$, $M_{22}$, $M_{23}$, $M_{24}$, $J_3$, $McL$, $Ru$, $O'N$, $Co_3$.
\end{prop}

The list in (2) may not be complete as the table on p.61-62 of \cite{AH} contains groups of relatively small orders.

One of the  goals of this paper is to obtain sharp upper bounds for    
$e(G)$ for finite simple groups $G$. Our results in this direction make a substantial progress.  One of the difficult cases
is that of $G=\PSU_n(q)$, for which our bound is $3n+3$ (Lemma \ref{pu2}). Other hard cases are 
the groups $G=E_6(q)$ and $G={}^2E_6(q)$, where our bound is hardly sharp. 

We summarize our results in the following theorem. Note that the cases of $G$ an alternating group or a sporadic simple group are settled in Theorem \ref{un4} and \ref{tt2}, so we assume $G$ of Lie type.

\begin{theor}\label{ag9}
Let $G=G(q)$ be a  finite simple group of Lie type with field parameter $q$, and let $p|q$ be a prime. Then $e(G)\leq b$, where b is as in Tables $1,2$.

\end{theor}

\subsection{The conjugacy class exponent of finite simple groups } 

Next we describe our results on $e_G(C)$ for individual conjugacy classes $C$ of simple groups $G$. For the alternating groups and the sporadic simple groups the situation is satisfactory as $e_G(C)\leq 3$ by Theorems \ref{un4} and \ref{tt2}, so $e_G(C)=3$
if and only if $C$ is not real. The real conjugacy classes of alternating groups are well described (see for instance \cite[Lemma 3.5]{SiZ} or elsewhere)
and those in the sporadic groups can be easily read off from their
character tables \cite{at}. So Problem 4 reduces to simple groups of Lie type.
Problems of determining real conjugacy classes in finite simple groups $G$ of Lie type are addressed in \cite{TZ05}, however, no full classification of real classes in $G$ is available so far. We mention \cite{GS}, which completed this for $\PSL_n(q)$, and \cite{SFV} 
for unitary groups $\PSU_n(q)$. 

The theory of conjugacy classes forms a significant part of general structure theory of groups  
  of Lie type. This is mainly based on the theory of algebraic groups,  see \cite[Ch E]{Sbo}, \cite[Ch. 8]{Hm} and focus on the parameterization of the conjugacy classes, in particular, 
in counting the classes which lie in the same conjugacy class of the corresponding algebraic group.  The book \cite{LiS} is 
entirely devoted  to the classification of the conjugacy classes of unipotent elements.

Some results of  \cite{TZ04} and \cite{TZ05} are very useful in our analysis, especially \cite[Theorem 1.5]{TZ04}.  
As in \cite{TZ04}, we first single out the cases  $e_G(g)$ for $g$ unipotent and semisimple.  
For a group $G$ of Lie type it is convenient to denote a semisimple element of $G$  by $s$ 
and a unipotent element  by $u$. Clearly, $e(G)=2$ implies $e(g)= 2$ for every $1\neq g\in G$. The cases with $e(G)=2$ are listed in Theorem 1.1; the cases with $e(u) \leq 2$ for every unipotent element $u \in G$ are listed in \cite[Theorem 1.4]{TZ05}.

\begin{table}[h]
	  \caption{The conjugacy class exponent of classical groups of Lie type}\label{Table_classical}
  {\fontsize{10}{12}\selectfont
  
\begin{center}
  \begin{tabular}{|c|c|c|c|}
      \hline
$G$ & $p|q$ & {\rm bound} & {\rm reference} \\ \hline
$\PSL_2(q)$ & $q \equiv 3 \pmod 4$ & $3$ & {\rm \mbox{\cite[Theorem 4.2, p. 240]{AH}}} \\
 & $q \not \equiv 3\pmod4$ & $2$ & {\rm Theorem ~\ref{un4}} \\ 
\hline
$\PSL_n(q)$, $n>2$ & $q>3$ & $n$ & {\rm Lemma ~\ref{Lev}} \\
$\PSL_n(2)$, $n>5$ & $q=2$ & $6$ & {\rm Lemma~ \ref{gg6}} \\
$\PSL_n(2)$, $n=3,4,5$ & $q=2$ & $3$ & {\rm {\rm Lemma ~\ref{s25}}} \\
$\PSL_n(3)$, $n>2$ & $q=3$ & $6$ & {\rm Lemma ~\ref{33s}} \\
$\PSL_n(4)$, $n>2$ & $q=4$ & $18$ & {\rm Lemma ~\ref{44s}} \\
 \hline
$\PSU_n(q)$, $n>2$ & $n=p^k$ & $n$ & {\rm Lemma~ \ref{5eg}} \\
$\PSU_n(q)$, $n>2$ & $(q,n)=1$ & $3n$ & {\rm Lemma~ \ref{pu2}} \\
$\PSU_n(q)$, $n>2$ & $(q,n)>1$ & $3n+3$ & {\rm Lemma~ \ref{pu2}} \\
$\PSU_n(2)$, $n>2$ & $p=2$ & $18$ & {\rm Lemma ~\ref{pu3}} \\
$\PSU_n(q)$, $n>2$ & $p=2$ & $2n-2$ & {\rm Theorem ~\ref{u36}} \\
$\PSU_3(q)$ & $q>2$ & $3$ & {\rm Lemmas ~\ref{ma3}, ~\ref{u3q1}} \\
$\PSU_4(q)$ & $p>2$ & $6$ & {\rm Lemma ~\ref{6ps}} \\
\hline
$\PSp_{2n}(q)$, $n>4$ & $q \equiv 3 \pmod 4$ & $6$ & {\rm Theorem ~\ref{symplectic}} \\
\hline
$\PSp_{4}(q)$ & $q\not\equiv 3\pmod4$ & $2$ & {\rm  Theorem ~\ref{un4}} \\
 & $q \equiv 3 \pmod 4$ & $3$ & {\rm Proposition ~\ref{s4q}} \\
\hline
$P\Omega^+_{4n}(q)$ & $q \not \equiv 3\pmod 4$ & $2$ & {\rm Theorem ~ \ref{un4}} \\
$P\Omega^+_{4n+2}(q)$ & & $4$ & {\rm Theorem ~\ref{orth4} } \\
$P\Omega^-_{4n+2}(q)$ & $p>2$ & $4$ & {\rm Theorem ~\ref{orth4} } \\
$P\Omega^-_{4n+2}(q)$ & $p=2$ & $8$ & {\rm Theorem ~\ref{orth4} } \\
$P\Omega^-_{4n}(q)$ & & $2$ & {\rm Theorem ~\ref{un4}} \\
$\Omega_{2n+1}(q)$ & $q \equiv 1 \pmod 4$ & $2$ & {\rm  Theorem~ \ref{un4}} \\
\hline
$\Omega_{9}(q)$ & $q \equiv 3 \pmod 4$ & $2$ & {\rm  Theorem ~\ref{un4}} \\
\hline
$P\Omega^+_{8}(q)$ & & $2$ & {\rm  Theorem ~\ref{un4}} \\
\hline
\end{tabular}\end{center}
  }
\end{table}

The following result is based on Corollary \ref{cor881} and \cite[Theorem 1.5]{TZ04} (quoted as Lemma \ref{u5n} below) and \cite[Theorem 1.4]{TZ05}. 

\begin{theor}\label{u55}
Let $G=G(q)$ be a finite simple group of Lie type, and let $u\in G$ be a unipotent element. 

$(1)$ If $q$ is odd then $e(u)\leq 3$ unless $G=E_8(q)$ and $5|q$, in which case $e(u)\leq 6.$

$(2)$ If $q$ is even and $G\in\{F_4(q)$, $E_7(q)\}$ then $e(u) \leq 4$. If $G=\Sz(q)$ then $e(u)\leq 3$.

$(3)$ If $G=E_8(q)$, $q$ even, then $e(u) \leq 8$.

\noindent $(5)$ In addition, $e(u)\leq 2$ if one of the following holds: 

$(i)$ $G$ is as in Theorem {\rm 1.1;}


$(ii)$ $q\not\equiv 3\pmod 4;$ 

$(iii)$ $G\in\{\PSL_n(q),\PSU_n(q)\}$,   $n$ odd; 

$(iv)$ $G=\PSL_n(q)$, $n/(n,q-1)$ even, or $\PSU_n(q)$, $n/(n,q+1)$ is even;  

$(v)$ $G=\Omega^\pm_n(q)$, $q\equiv 3\pmod 4$ and either $n\in \{7,16,17,24\}$ or  $n$ is even and $\mp 1=(-1)^{n/2}.$

$(vi)$ $p \neq 3$ and $G \in \{G_2(q),E_6(q),{}^2E_6(q)\}$ or $G=F_4(q),p>3$.

$(vii)$  $q$ is even and $G \not \in \{\Sz(q)={}^2B_2(q),F_4(q), E_7(q),E_8(q)\}.$ 
\end{theor}

Observe that Corollary \ref{cor881} is stated for an arbitrary finite group $G$.
Assuming that  $g\in G$ is of odd prime power order and the generators of the cyclic group $\langle g\rangle$ lie in at most two conjugacy classes of $G$, it claims that   $e(g)\leq3$.  In turn,  \cite[Theorem 1.5]{TZ04} claims that this assumption holds (with a few exceptions) 
for the unipotent elements of finite groups of Lie type.

This motivates our interest to elements of the above kind. We focus on such elements $g\in G$ with $e(g)>2$. Specifically, we define:

\begin{defi} \label{defsemirational}
Let $G$ be a finite group and $g\in G$. We say that $g$ is semirational if every generator of $\langle g \rangle$ is conjugate to $g$ or $g^{-1}$.
\end{defi}

The term ``semirational'' was introduced in \cite{cd}. Note that such elements play some role in the theory of finite groups, see \cite{Mo}. In Section 3 we discuss $e_G(g)$ for an arbitrary finite group $G$ and elements $g\in G$ not necessarily of prime power order.  

It is known that, for a finite group $K$ and $g\in K$, the group $N:=N_K(\langle g\rangle)/C_K(\langle g\rangle)$ $\subseteq {\rm Aut }\,K$ determines the field $\mathbb{Q}(g)$  spanned by $\chi(g)$ when $\chi$ ranges over the characters of $K$. In particular, if  $N \cong {\rm Aut }\,K$ then $\mathbb{Q}(g)=\mathbb{Q}$. This motivated a detailed study of 
the index $|{\rm Aut }(K):N|$ when $K=\langle g \rangle$ in \cite{TZ04} and \cite{Ge3}.  

\begin{table}[ht]
  \centering
	  \caption{The conjugacy class exponent of exceptional groups of Lie type}\label{Table_exceptional}
  {\fontsize{10}{12}\selectfont
    \begin{tabular}{|c|c|c|c|}
      \hline
$G$ & $p$ & {\rm bound}&{\rm reference}\\ \hline
${}^3D_4(q)$ & & $2$ & {\rm Theorem\ \ref{un4}} \\ \hline
$E_6(q)$, ${}^2E_6(q)$ & $p>11$  &{\rm \ $\min\{p,36\}$}& {\rm Theorem\ \ref{666}}\\
 & $p=2$ & $16$ & {\rm Lemma\ \ref{e62}} \\
 & $p=3$ & $27$ & {\rm Lemma\ \ref{e62}} \\
 & $p=5$ & $25$ & {\rm Lemma\ \ref{e62}} \\
& $p=7,11$ & $36$ & {\rm Theorem\ \ref{666}} \\
\hline
$E_7(q)$& $p>2$ & $6$ & {\rm Lemma\ \ref{ee7}} \\
 & $p=2$ & $4$ & {\rm Theorem\ \ref{ue67}} \\ 
\hline
$E_8(q)$ & $p\neq 2,5$ & $6$ & {\rm Lemma\ \ref{ee7}} \\
 & $p=5$ & $6$ & {\rm Theorem\ \ref{ue67}} \\
 & $p=2$ & $8$ & {\rm Theorem\ \ref{ue67}} \\
\hline
$F_4(q)$ & $p>2$ & $6$ & {\rm Lemma\ \ref{ee7}} \\
 & $p=2$ & $4$ & {\rm Theorem\ \ref{sf4}} \\
\hline
${}^2F_4(q)$& $p=2$ & $4$ & {\rm Theorem\ \ref{sf4}} \\
\hline
${}^2F_4(2)'$ & $p=2$ & $3$ & \rm {Theorem\ \ref{2sf4}} \\
\hline

$G_2(q)$ & $p>2$ & $6$ & {\rm Lemma\ \ref{ee7}} \\
 & $p=2$ & $3$ & {\rm Lemma\ \ref{g24}} \\
\hline
${}^2G_2(q)$ & $p=3$ & $3$ & \rm Lemma ~\ref{ma3} \\
\hline
${}^2B_2(q)$ & $p=2$ & $3$ & \rm Theorem\ \ref{sf4} \\
\hline
\end{tabular}
  }
\end{table}
 
If  $s\in G$ is semisimple then $s$ is real unless possibly for $G$ in the following list: $\PSL_n(q)$, $n>2$; $\PSU_n(q)$, $n>2$; $\Omega_{2n}^{\pm}(q)$, $n$ odd; $E_6(q)$; ${^2}E_6(q)$ (Lemma \ref{sn2}).
For these  groups  we obtain the following bounds:

\vskip2cm

\begin{theor}\label{s67} Let $s\in G$ be a non-real semisimple element.

$(1)$ If $G=\PSL_n(q)$, $n>1$, then $e(s)\leq n$,  
and $e(s)\leq3,6,9$ if $q=2,3,4$, respectively; 

$(2)$ Let $G=\PSU_n(q)$, $n>2$. If either $(n,q)=1$ or $n$ is a $p$-prime power for a prime $p|q$, then $e(s)\leq n$. If $(n,q) > 1$, then $e(s)\leq n+1;$ 

$(3)$ If $G=\Omega^\pm_{2n}(q)$, $n>3$ odd, then $e(s)\leq 4;$

$(4)$ If  q is odd  and $G=E_6(q) $ or  ${}^2E_6(q)$, then $e(s)\leq 13.$ 
\end{theor}

{\it Notation}. For integers $m,n$ we write for $(m,n)$ for the greatest common divisor of $m,n$. $\diag(A_1,\ldots,A_k)$ a block-diagonal matrix with diagonal blocks $A_1,\ldots,A_k$.

 If $F$ is a field $F^\times$ means the multiplicative group of $F$. $\mathbb{F}_q$ 
a finite field of $q$ elements.  $\overline{\mathbb{F}}_q$ the algebraic closure of  $\mathbb{F}_q$. 

If $G$ is a group and $g\in G$ then $|g|$ is the order of $g$ and  $\langle g \rangle$ is the cyclic group generated by $g$. We denote by $Z(G), G'$ and $O_p(G) $ the center, the derived subgroup and the maximal normal $p$-subgroup of $G$.
 By $\Phi(g)$ we denote the set of generators of $\langle g \rangle$. If all elements of $\Phi(g)$ are conjugate then $g$ is called rational (in $G$). 

If $x,y$ are elements of a group $G$ then $x^y$ means $x^{-1}yx$ and we write $x \sim y$ if   they are conjugate in $G$.  If $x\sim  x^{-1}$ then $x$ is called real (in $G$).

If $g \in G$, we denote by $|g|$ the (usual) order of $g$ and by $e(g)=e_G(g)$ generalized order, that is, the smallest positive integer $k$ such that $1 \in C^k$, where $C$ is the conjugacy class of $g$ in $G$. We call $k$ the exponent of $C$.

We denote by $J_n$ (and refer to it as a Jordan block) an upper triangular matrix with 1 at its diagonal and just above it, 0 elsewhere. In other words 
$$(J_n)_{ij} = \left\{ \begin{array}{ll} 1 & \mbox{if } i=j \mbox{ or } i+1=j, \\
0 & \mbox{otherwise.} \end{array} \right.$$
\indent We denote by $S_n$ and $A_n$ the symmetric and alternating groups on $n$ letters, respectively. Our notation for simple groups of Lie type are standard, for sporadic groups we follow \cite{at}. 

If $G$ is a group of Lie type in defining characteristic $p$ then
$p$-elements are called unipotent and $p'$-elements are called semisimple. In addition, 
if every $g\in G$ decomposes as   $g=su$, where $s,u\in\langle g \rangle$, $s$ is semisimle and $u$ is unipotent. In algebraic group theory this is referred as a {\it Jordan decomposition} of $g$
\cite[Section 2]{MTe}.

\section{Some known facts and initial observations\\ on conjugacy class exponents}

\subsection{General remarks}

\begin{lemma}\label{gr1}
Let $G$ be a group and let $x,y,z\in G$. 

$(1)$  $e(xy) \leq \mbox{l.c.m.}(|x|,|y|)$.

$(2)$   If $|x|=2$ and $y^a=y^{-1}$ for some $a \in $G, then $e(xy) \leq 4$.

$(3)$   If $|x|=|y|=|z|\in G=2$, then $e(xyz) \leq 4$.
\end{lemma}

\begin{proof}
(1) Let $g=xy$ and let $|x|=n$, $|y|=m$. Then 
$$h:= g \cdot g^x \cdot g^{x^2} \cdots g^{x^{m-1}} = x \cdot y^m \cdot x^{m-1} = x^m.$$ 
Writing $l = \mbox{l.c.m.}(n,m)=m \cdot s$, we obtain that $h^s = x^{ms} = x^l = 1$, therefore $e(g) \leq l$.

(2) Let $g=xy$. Then $g \cdot g^x \cdot g^a \cdot g^{xa} = 1$.

(3) Let $g=xyz$. Then
$g \cdot g^z \cdot g^x \cdot g^{zx} = 1$.
\end{proof}

\begin{lemma}\label{xy8}
Let $x,y \in G$ be such that $xy=yx$. Let $m := e_{C_G(x)}(y)$, $n:=e_G(x^m)$. Then $e_G(xy) \leq n \cdot m$.
\end{lemma}

\begin{proof}
We have $\prod_{i=1}^m y^{t_i} = 1$ with $t_i \in C_G(x)$ for all $i$. Then $\prod_{i=1}^m (xy)^{t_i} = \prod_{i=1}^m x \cdot y^{t_i} = x^m$. We have $\prod_{j=1}^n (x^m)^{r_j} = 1$ with $r_j \in G$ for all $j$. Therefore 
$$\prod_{j=1}^n \left( \prod_{i=1}^m (xy)^{t_ir_j} \right) = \prod_{j=1}^n \left( \prod_{i=1}^m (xy)^{t_i} \right)^{r_j} = \prod_{j=1}^n (x^m)^{r_j} = 1$$
hence $e_G(xy) \leq n \cdot m$.
\end{proof}

\begin{lemma}\label{eq} Let G be a finite group and let $\mbox{{\rm Irr}}(G)$ denote the set of all the complex irreducible characters of $G$. If  C is the conjugacy class and $g\in G$  then $1\in C^{k}$  if and only if

\begin{equation}\label{eq1}\sum_{\chi \in \mbox{{\rm Irr}}(G)} \frac{\chi(x)^k}{\chi(1)^{k-2}}=1+\sum_{\chi\neq 1}\frac{\chi(x)^k}{\chi(1)^{k-2}}  \neq 0.\end{equation}
\end{lemma}

See \cite[Lemma 10.10]{ASH} and \cite[Equation (1)]{Shalev}.

\begin{lemma} \label{isomodules}
Let $G = \langle s \rangle$ be a finite cyclic group and let $\phi,\tau$ be two finite dimensional irreducible representations of $G$ over a finite field $F$. If the matrices $\phi(s)$, $\tau(s)$ have an eigenvalue in common, then $\phi$ and $\tau$ are equivalent, i.e. the corresponding $F[G]$-modules are isomorphic.
\end{lemma}

\begin{proof}
Let $F_1$ be the field $F(\zeta)$ where $\zeta$ is a primitive $|G|$-root of unity. Then, over $F_1$, $\phi$ is a direct sum of absolutely irreducible constituents that are obtained from each other by Galois conjugations by elements of ${\rm Gal}(F_1/F)$. We can say that they form a Galois group orbit. See \cite[Theorem 19.4(ii)]{Feit}. The same is true for $\tau$, so the orbits either coincide or are disjoint.

Note that two absolutely irreducible representations $\mu$, $\nu$ of a cyclic group $\langle s \rangle$ coincide if and only if $\mu(s)=\nu(s)$. Since $\phi(s)$, $\tau(s)$ have an eigenvalue in common, and this eigenvalue belongs to $F_1$, $\phi$ and $\tau$ share a $1$-dimensional constituent, so the orbits are the same. It follows that $\phi$ and $\tau$ have the same irreducible constituents over $F_1$, so $\phi(s)$, $\tau(s)$ have the same list of eigenvalues up to ordering.
So $\phi$ and $\tau$ are equivalent in $\GL(W)$, where $W$ is the common underlying space of $\phi,\tau$ over $F_1$. So $\tau(s) = g \phi(s) g^{-1}$ for some $g \in \GL_m(F_1)=\GL(W)$. In terms of matrix entries, the equation $g\phi(s) = \tau(s)g$ can be viewed as a homogeneous system of linear equations with indeterminates $g_{ij}$, $1 \leq i,j \leq m$ and coefficients in $F$. Then this system has a non-zero solution over $F$, so there is a matrix $h$ over $F$ such that $h\phi(s) = \tau(s)h$. Note that $\det(h) \neq 0$. Indeed, otherwise $Y:=hW=Y\neq W$ and hence $\tau(s)Y\subset \tau(s) h W = h\phi(s)W \subset hW=Y$, which is false as $\tau$ is irreducible.
\end{proof}

We   often use the following fact: if $V$ is a finite dimensional $F$-vector space, where $F$ is a finite field, then  any irreducible element of the group $\GL(V)$ is regular semisimple. (Indeed, if $s$ is a irreducible element and $u$ is a unipotent element commuting with $s$, then $u-1$ induces a $F\langle s \rangle$-morphism $V \to V$ that is not an isomorphism, so $u-1=0$ by Schur's lemma, in other words $u=1$.)

\begin{lemma}\label{in3} Let $H=\GL_n(q)=\GL(V)$ and let $h=su\in H$ be indecomposable. 
Then $h$ is contained in a subgroup $M = M_1 M_2$, $[M_1,M_2] = 1$, $M_1 \cong \GL_k(q)$, $M_2 \cong \GL_l(q)$ for some integers $k,l$ with $kl=n$, and $s\in M_1,u\in M_2$. (In fact, $M=\GL_k(q)\otimes \GL_l(q)$, the Kronecker product.) In addition, $s$ is irreducible as an element of $\GL_k(q)$ and $u$ represents a Jordan block $J_l$ in $\GL_l(q)$.\end{lemma}

\begin{proof} Let $V = \bigoplus_{i=1}^t V_i$, where $V_1,\ldots,V_t$ are homogeneous components of $s$ on $V$, i.e. each $V_i$ is the maximal direct sums of pairwise $\mathbb{F}_q \langle s \rangle$-isomorphic irreducible components. Since $u$ induces a $\langle s \rangle$-invariant isomorphism on each irreducible component, $uV_i=V_i$ for $i=1,\ldots, t$. As $h$ is indecomposable, we have $t=1$.
Therefore, $V$ is a direct sum of pairwise isomorphic irreducible ${\mathbb F}_q\langle s\rangle $-submodules of equal dimension $k$, say. It is well known (by Clifford theory and Schur's lemma) that $C_H(s)\cong \GL_l(q^k)$, where
$l$ is the composition length of $V$ and $k=n/l$. Then the conjugacy class of $u$  in $\GL_l(q^k)$ meets $\GL_l(q)$, since $u$ can be taken to its canonical Jordan form. We denote this group by $M_2$ and set $M_1=C_{\GL(V)}(M_2)$. We can view $M_1$ as the group $\diag(A,\ldots,  A)$ with $A \in \GL_k(q)$ and $M_2$ as $C_{\GL(V)}(M_1)$. A conjugate of $h$ lies in $M = M_1M_2$.
\end{proof}

Note that $\SL_k(q)\otimes \SL_l(q)\subseteq M\cap \SL(V)$, and  $u\in \SL_l(q)$.
However, if $h\in \SL_n(q)$ then $s$ is not necessarily contained in $\SL_k(q)$ as $\det \diag(A,\ldots, A)=(\det A)^l=\det A^l$ and hence $s^l\in \SL_k(q)$ does not imply $s\in \SL_k(q)$.

\begin{lemma}\label{centsem}
See \cite[Lemma 3.2]{TZ04} Let $H =\GL_n(q)=\GL(V)$.
The centralizer of a semisimple element $s\in H$ is described as follows. Let $V=\bigoplus_{i=1}^k V_i$, where the $V_i$'s are homogeneous $\mathbb{F}_q\langle s\rangle$-modules, and let $s_i$ be the projection of $s$ into $H_i = \GL(V_i)$. Then $C_H(s) V_i = V_i$ for all $i$ and $C_H(s)$ is the direct product $\prod_{i=1}^k C_{H_i}(s_i)$. Let $m_i$ be the common dimension of the irreducible constituents of $s$ on $V_i$. Then $C_{H_i}(s_i) \cong \GL_{n_i/m_i}(q^{m_i})$, where $n_i = \dim V_i$. 
\end{lemma}

\subsection{Some known facts about finite groups of Lie type}

Recall that a group of Lie type $G$ is defined as $C_{\mathbf{G}}(\si)$, where   $\mathbf{G}$ is a reductive connected algebraic group and $\sigma:\mathbf{G}\rightarrow\mathbf{G}$
is an algebraic group homomorphism such that  $G:=C_{\mathbf{G}}(\sigma)$ is finite \cite[Theorem 21.5]{MTe}.
A standard notation for $C_{\mathbf{G}}(\sigma)$ is $\mathbf{G}^\sigma$. Then $\si$
is called a Frobenius (or Steinberg) endomorphism of $\mathbf{G}$.  

All simple groups of Lie type are of the form $G'/Z(G')$ for $\mathbf{G} $ a simple algebraic group with some exceptions 
(given in \cite[Theorem 24.17]{MTe}). Let $G$ be a group of Lie type (often referred as a finite reductive group) and let $p$ be the defining characteristic of $G$.  
We define a regular semisimple element $s\in G$ as one with $C_G(s)$ a $p'$-group;   the latter is equivalent to saying that  $C_{\mathbf{G}}(s)$ has no non-trivial unipotent  element
(see for instance \cite[\S 2.3]{Hm} and \cite[Proposition  14.23]{DM}).  In particular, if
$g\in G=\mathbf{G}^\sigma$ is a regular semisimple element then $g$ is regular semisimple
in $\mathbf{G}$ and in $\mathbf{G}^{\sigma^k}$ for $k>1$.
 
The group $N_G(S)$ with $S$ a Sylow $p$-subgroup of $G$ is called a Borel subgroup of $G$. Every subgroup $P$ such that $B\subseteq P\subset G$ is called parabolic. 
We have $P=L\cdot O_p(P)$, a semidirect product, and $L$ is called a {\it Levi subgroup} of $P$.

  
 \begin{lemma}\label{Lev} {\rm \cite{Lev1996}} 
Let $G$ be the group $\PSL_n(F)$, where $n \geq 3$, $F$ is a field and $|F| \geq 4$. Assume, further, that if $n=3$, then $F$ is either finite or algebraically closed. Then $\cn(G)=n$.
\end{lemma}

\begin{lemma}\label{ma3} \cite[Lemma 2.3, Propositions 3.5, 3.6, 3.7]{Malcolm21} 
Let $G$ be one of the groups $A_n, n>4$, $\PSL_3(q)$, $\PSU_3(q)$, $(3,q+1)=1$, ${}^2G_2(3^{2m+1})$, $m>0$. Then every element of $G$ is a product of two elements of order $3$. Consequently, $e(G)\leq 3$ for these groups.
\end{lemma}

The second claim of the above lemma follows from Lemma \ref{gr1}(1).

\begin{lemma}\label{u5n} {\rm \cite[Theorem 1.5]{TZ04}} Let $\mathbf{G}$ be a simple algebraic group in characteristic $p>0$, $\si$ a Frobenius endomorphism of $\mathbf{G}$ and $G=\mathbf{G}^\si$ be a finite group of 
Lie type. Suppose that G is not a Suzuki group, and $p> 2$ if $\mathbf{G}$ is of type   
$F_4,E_7$ and $p\neq 2,5$ if $\mathbf{G}$ is of type   
$E_8$. Let $g=su\in G$ , where $u$ is a p-element, s is a $p'$-element and $us=su$.
Then the elements $u^i$ with $(i,p)=1$ lie in at most two $C_G(s)$-conjugacy classes. 
\end{lemma} 

Observe that in  \cite[Def 2.4]{TZ04} semirational $p$-elements  are called half-rational if $p>2$ and  half-*-rational for $p=2$. Indeed, \cite[Def 2.4]{TZ04} states that $g$ is called half-*-rational if $g$ is conjugate to $g^m$ whenever $4|(m-1)$. In particular,
$g$ is not conjugate to $g^{-1}$, and if $g^m$ is not conjugate to $g$ then $4|(m+1)$; then $4|(-m+1)$ so $ g^{-m}$ is not conjugate to $g$.

The  following lemma is a special case of a result by Gow \cite[Theorem 2]{Gow};

\begin{lemma}\label{Gow}  Let $G$ be a simple group of Lie type,
$g\in G$ a regular semisimple element and let $C$ be the conjugacy class containing g. Then $C^2$ contains all non-identity semisimple elements of $G$. Consequently, $e(g)\leq 3$.
\end{lemma}

This result is complemented  in \cite[Theorem 1.8]{GuT} as follows:

 \begin{lemma}\label{GuT1}   Let $(G,S)$ be one of the following pairs of groups: 

$(1)$ $G=\GL_n(q)$, $n\geq 2$,  $(n,q)\neq (2,2),(2,3)$ and $S=\SL_n(q);$

$(2)$ $G=\U_n(q)$, $n\geq 2$,  $(n,q)\neq (2,2),(2,3),(3,2)$ and $S=\SU_n(q).$

\noindent Let $h\in G$ be a regular semisimple element and let $g\in G$ be a non-central semisimple element such that $g\in h^2S$. Then $g$ is a product of two conjugates of $h$. 
\end{lemma}

\begin{corol}\label{GuT} Let $(G,S)$ and h be  as in Lemma {\rm \ref{GuT1}}. Suppose that $h^{-1}\in h^2SZ(G)$. Then  there are two elements   $x,y\in S$ such that $hxhx^{-1}yhy^{-1}\in Z(G)$.\end{corol}

\begin{proof} It is well known that $\det C_G(s)$ contains all elements of $\GL_1(q) $ 
 or
$\U_1(q)$ for a semisimple element $s$ of $G$ and $G=\GL_n(q)$ or $\U_n(q)$,  respectively. It follows that   
the $S$- and $G$-conjugacy classes of semisimple elements  of $G$ coincide. In addition, $zh^{-1}\in h^2S$ for some $z\in Z(G)$. So, by  Lemma  \ref{GuT1}, $zh^{-1}=xhx^{-1}yhy^{-1}$ for     some  $x,y\in S$. So the result follows.\end{proof}

\begin{lemma}\label{5eg} {\rm \cite[Theorem H, p.344]{egh}}
Let $G$ be a simple group of Lie type in defining characteristic $p>0$, and $g\in G$. Then $g$ is a product of two unipotent elements. Consequently, $e(G)\leq d$, where $d$ is the exponent of 
a Sylow p-subgroup of $G$.
\end{lemma}

The second claim follows from the first one and Lemma \ref{gr1}(1). 

\begin{lemma}\label{un2} Let $G$ be one of the groups $\PSL_n(q)$, $\PSU_n(q)$ with $n>1$. Suppose that $(n,q)=1$. 
Then every semisimple element $s\in G$ is a product of two elements of order $n$. %
Consequently, $e(s)\leq n$. \end{lemma}

\begin{proof} We consider only the case of $G=\PSU_n(q)$, the other case is similar.

Let $H=\SU_n(q)=\SU(V)$    and let $h\in \U(V)$
be a matrix  permuting cyclicly the elements of an orthonormal basis of $V$.
If $n$ is odd then $\det h=1$. If $n$ is even then $q$ is odd and $t:=\diag(-1,\Id_{n-1})\in \U(V)$ (under the same basis). Then $(th)^n=-\Id$ is scalar. Let $g\in G$ be the projection of 
$ h$. Then $|g|=n$. Note that  $h$ is regular in $H$ as $h$ has $n$ distinct eigenvalues, since $(n,q)=1$.
Therefore, $g$ is regular in $G$. By Lemma \ref{Gow}, $s$ is a product of two conjugates of $h$. By Lemma \ref{gr1}(1), $e(g)\leq n$.    
\end{proof}

\begin{lemma}\label{u42} Let $G=\PSU_n(q),n>1$. Suppose that    $(n,q)>1$.  
Then every semisimple element $s\in G$ is a product of two elements of order   $n+1$. 
Consequently, $e(s)\leq n+1$. \end{lemma}

\begin{proof}  
Note that $(n+ 1,q)=1$. Let $x\in \U_{n+1}(q) = \U_{n+1}(V)$ be the matrix permuting cyclically an orthonormal basis of $V$. Then $x\in S_{n+1}$, the group of all permutations of the basis.   Let $\mu=(-1)^{q}. $ As $(n+1,q)=1$,   
the multiplicity of the eigenvalue $\mu$ of $x$ equals 1. So the $\mu$-eigenspace $W$ of $x$ is non-degenerate, and hence $V=W^\perp\oplus W$. 
As $\det x=\mu$, we have $x=\diag(h,\mu)$,
where $h\in \SU_n(q)$. Then $h$ is regular since $h$ has $n$ distinct eigenvalues over $\overline{\mathbb{F}}_q$.
Let $g$ be the projection of $h$ in $G$. Then $|g|=n+1$.   
By Lemma \ref{Gow}, $s$ is a product of two conjugates of $g$. By Lemma \ref{gr1}, $e(s)\leq n+1$.\end{proof}

\begin{lemma}\label{sn2} {\rm \cite[Proposition  3.1(ii)]{TZ05}} Let $G$ be a simple group of Lie type.
Suppose that $G$ is not in the following list: $\PSL_n(q)$, $n>2$; $\PSU_n(q)$, $n>2$; $\Omega^\pm_{2n}(q)$, $n$ odd; $E_6(q)$; ${}^2E_6(q)$. Then every semisimple element of $G$ is real.
\end{lemma}

\begin{lemma} \label{ev1} {\rm \cite[Theorem 1.8]{TZ04}} Suppose that q is even and let $G$ be a group in the following list: $\SL_n(q)$, $\SU_n(q)$, $\Sp_{2n}(q)$, $\Omega^\pm _{2n}(q)$, ${}^3D_4(q)$, $E_6(q)$, ${}^2E_6(q)$, $G_2(q)$. Then the unipotent elements of $G$ are rational.
\end{lemma}

\begin{lemma}\label{ru5} Let $G$ be a simple group of Lie type in defining characteristic $p>0$, and $g\in G$. Suppose
that $e(g)>3$. Then $g\in P$ for some maximal parabolic subgroup P of G and $O_p(\langle g\rangle)\subseteq O_p(P)$. In addition, the elements of $O_p(P)$ are not regular unipotent unless $G$ is in the following list: $\PSL_2(q)$, $\Sz(q) = {^2}B_2(q)$, ${^2}G_2(q)$.
\end{lemma}

\begin{proof} By Lemma \ref{Gow},  $g$ is not a regular semisimple element of $G$. 
Therefore, $C_G(g)$ contains a unipotent element $u\neq 1$, say, and we choose $u$ to be a generator of $ O_p(\langle g\rangle)$ whenever     this group is non-trivial.

 Set $K= \langle g\rangle$. By a Borel-Tits theorem \cite[Theorem 26.5]{MTe}, %
$K$ is contained in a parabolic subgroup 
$P$, say, of $G$, and $O_p(K)\subseteq O_p(P)$. In particular, $u\in O_p(P)$.

For the additional claim, suppose the contrary, that $v\in U:=O_p(P)$ for some regular unipotent element $v$ of $G$.  If $L:=P/U$ contains no unipotent element then $P$ is a Borel subgroup of $G$ \cite[\S 26.1]{MTe}.  In this case $G$ is of BN-pair rank 1, and hence $G\in \{\PSL_2(q), \Sz(q) = {^2}B_2(q), {^2}G_2(q)\}$.
Otherwise,  the projection of $v$ into $P/U$ is non-trivial by \cite[Lemma 2.6]{tz13},
which is a contradiction.\end{proof}

Let $G$ be a finite simple group of Lie type in defining characteristic $p$.
 In the proof of the following lemma we use \cite[Theorem 1.5]{TZ04} which requires $p$, the defining characteristic of $G$, to be almost good. According to \cite[p. 328]{TZ04}, $p$ is almost good for $G$ unless $p=2$ and $G=\Sz(q)$,  $F_4(q)$, ${}^2F_4(q)$, $E_7(q)$, and $p=2,5$ for $E_8(q)$. 

\begin{lemma} \label{ee7} Let $G$ be a finite simple group of Lie type in characteristic $p>2$. Then $e(G) \leq 6$ unless possibly one of the following holds:

$(1)$ $G = E_6(q)$, ${^2}E_6(q);$

$(2)$ $G=E_8(q)$, $p=5;$

$(3)$ $G=\PSL_n(q)$, $n > 2;$

$(4)$ $G=\PSU_n(q)$, $n > 2$.

\noindent
In particular,  if $q$ is odd and
$G$ is $\Omega^\pm_{2n}(q)$ with $n$ even, $\Omega_{2n+1}(q)$ or $\PSp(2n,q)$,   
then $e(G)\leq 6$.
\end{lemma}

\begin{proof}
Let $g=su$, where $s$ is semisimple, $u$ is unipotent and $s,u\in \langle g \rangle$. By Lemma \ref{sn2}, $s$ is real.   Recall that $p>2$ is almost good for $G$ except $p=5$ for $E_8(q)$.    Then,  by  Lemma \ref{u5n},  
the generators of $\langle u \rangle$ lie in at most two conjugacy classes of $C_G(s)$. By Corollary \ref{cor881}, the generalized order of $u$ in $C_G(s)$ is at most $3$. Since $s$ is real, $s^3$ is real too, so Lemma \ref{xy8} implies that $e(g) \leq 2 \cdot 3 = 6$.
\end{proof}

\section{Conjugacy class exponent of semirational elements }

Our interest in semirational elements (see Definition \ref{defsemirational}) is derived from the fact that the overwhelming majority of unipotent elements in groups of Lie type are semirational (Lemma \ref{u5n}). These c

In this section $G$ is an arbitrary finite group.
 
\begin{lemma} \label{i1i2i}
Let $g \in G$ be a semirational element of order $n$. Let 
$$\begin{array}{l}
I=\{m\ |\ 1 \leq m \leq n,\ (m,n)=1\}, \\
I_1=\{m \in I\ |\ g \sim g^m\}, \\
I_2=\{m \in I\ |\ g^{-1} \sim g^m\}.
\end{array}$$
We identify each integer with its reduction modulo $n$. Then $1 \in I_1$, $-1 \in I_2$ and $I_1 \cup I_2 = I$. If $h$ is a positive integer such that $e(g) > h+1$ then, for every $m_1,\ldots,m_h \in I_1$ such that $m_1+\cdots+m_h \in I$, we have $m_1+\ldots+m_h \in I_1$.\end{lemma}

\begin{proof}
Assume $m_1,\ldots,m_h \in I_1$ satisfy $m_1+\cdots+m_h \in I \setminus I_1$. By assumption $m_1+\cdots+m_h \in I_2$ so $m=-m_1-\cdots-m_h \in I_1$ and $g^{m_1} \cdots g^{m_h} g^m = 1$. Since $m_1,\ldots,m_h,m \in I_1$, this implies $e(g) \leq h+1$, a contradiction.\end{proof}

\begin{prop} \label{semir1}
Let $g \in G$ be a semirational element of order $n$. Write $n=\prod_{i=1}^k p_i^{a_i}$ with the $p_i$ pairwise distinct primes. If $p_i \geq k+2$ for all $i=1,\ldots,k$, then $e(g) \leq 3$.\end{prop}

So for example if $n=5^a 7^b 11^c$ then $e(g) \leq 3$.

\begin{proof}
Assume by contradiction that $e(g) > 3$. We use the notation of Lemma \ref{i1i2i} and, in particular, we identify each integer in $I$ with its residue modulo $n$. If $I_1=I$ then $e(g) \leq 2$, so we may assume that $I_1 \neq I$. Since $p_i \geq k+2$ for all $i$, we have $1,2,\ldots,k+1 \in I$. By Lemma \ref{i1i2i} we may assume that, whenever $a,b \in I_1$ are such that $a+b \in I$, we have $a+b \in I_1$. So, since $1 \in I_1$ and $1,2,\ldots,k+1 \in I$, we have $2=1+1 \in I_1$, $3=2+1 \in I_1$ and so on until $k+1$, so $1,\ldots,k+1 \in I_1$. If $m \in \mathbb{N}$ then $I \cap \{m+1,\ldots,m+k+1\} \neq \varnothing$. Indeed, if by contradiction this is false, then for all $i=1,\ldots,k+1$ there exists a prime $p_{j(i)} \in \{p_1,\ldots,p_k\}$ which divides $m+i$. Clearly, there exist two distinct $i_1,i_2 \in \{1,\ldots,k+1\}$ such that $p_{j(i_1)} = p_{j(i_2)}$, in other words there exists a prime $p$ dividing $n$ and dividing $m+i_1$, $m+i_2$ for some $i_1,i_2$, so $p$ divides $i_1-i_2$. This is impossible because $|i_1-i_2| \leq k$ and $p > k$. Therefore there exists $f(m) \in \{1,\ldots,k+1\} \subseteq I_1$ such that $m+i \not \in I$ for all $i$ with $0 < i < f(m)$ and $m+f(m) \in I$. Define a sequence $(\ell_k)_{k \geq 1}$ by $\ell_1:=1$ and $\ell_{k+1} := \ell_k+f(\ell_k)$. Since $1 \in I_1$, we have $\ell_k \in I_1$ for all $k \geq 1$ by Lemma \ref{i1i2i}, on the other hand it is clear that $\{\ell_k\ |\ k \geq 1\} = I$ so $I_1=I$, a contradiction.
\end{proof}

For a positive integer $n$, set $j(n)$ to be $\max_{m \in \mathbb{N}} (a_{m+1}-a_m)$ where $\{a_m\}$ is the increasing sequence of all positive integers coprime to $n$.
In other words, $j(n)$ is the maximal distance between ``two consecutive numbers coprime to $n$'', where $a,b$ are called ``consecutive coprime to $n$'' if they are coprime to $n$ (meaning that $(a,n)=(b,n)=1$) and no integer $m$ such that $a<m<b$ is coprime to $n$. Note that $j(n)$ is well defined because it can be calculated working modulo $n$. For example, if $n$ is a prime power, then $j(n)=2$. Observe that, if $n=\prod_{i=1}^k p_i^{a_i}$ where each $p_i$ is a prime number, then $j(n)=j(p_1 \ldots p_k)$.  

Note that $j(n)$ equals the smallest positive integer $m$ such that every sequence of $m$ consecutive positive integers contains an integer coprime to $n$. The function $j$ is called the \textit{Jacobsthal function} and there is literature about it. It is unbounded since, denoting by $p_i$ the $i$-th prime number (starting from $p_1=2$) we have $j(p_1 \cdots p_n) \geq p_n$ for every $n \geq 1$.

\begin{prop} \label{semir2}
Let $g \in G$ be a semirational element of order $n$. Then $e(g) \leq j(n)+2$.
\end{prop}
For example, if $n=3^a 5^b$, then $e(g) \leq j(n)+2=5$.
\begin{proof}
We use Lemma \ref{i1i2i} and the notation therein. We may assume that $I_1 \neq I$ (otherwise $e(g) \leq 2$) and that the conclusion of the lemma holds for $h=j(n)+1$. Let $x,y \in I$, with $x \in I_1$, and assume that $x,y$ are ``consecutive'' in $I$, meaning that $x<y$ and no integer strictly between $x$ and $y$ belongs to $I$. If we show that $y \in I_1$ we are done, because iterating this argument (starting from $x=1$) leads to $I_1=I$, a contradiction. We have $y-x \leq j(n)$ by definition of $j(n)$. Since $1 \in I_1$, Lemma \ref{i1i2i} implies that $y=x+(y-x) = x+1+\ldots+1 \in I_1$ since the number of summands on the right-hand side is at most $j(n)+1$.
\end{proof}

\begin{corol} \label{cor881}
Let $g \in G$ be a semirational element of order $n$. If $n$ is an odd prime power, then $e(g) \leq 3$. If $n$ is a power of $2$ then $e(g) \leq 4$.
\end{corol}

\begin{proof}
If $n$ is an odd prime power, we apply Proposition \ref{semir1} with $k=1$. If $n$ is a power of $2$, then Proposition \ref{semir2} implies that $e(g) \leq j(n)+2 = 4$.
\end{proof}
Note that this is best possible since, if $n \in \{3,4\}$ and $g \in G$ is a central element of order $n$, then $g$ is semirational and $e(g)=n$.

In this context it can be interesting to compute  $e_N(g)$, where $N=N_G(\langle g\rangle)$ and $g$ is semirational.  By Proposition \ref{semir2}, we have  $e_N(g)\leq j(|g|)+2$, however this bound is not best possible in view of Proposition \ref{semir1}. One may ask if $e_N(g)$ is always bounded from above by a constant.

\section{Sporadic groups}

For sporadic simple groups we use the notation of \cite{at}.

\begin{theor} \label{sporadicthm}
Let $G$ be a sporadic simple  group. If $G \in \{J_1,J_2\}$ then $e(G)=2$. Otherwise, $e(G)=3$.
\end{theor}

\begin{proof}
By \cite{TZ05}, the only simple sporadic groups in which all elements are real are $J_1$ and $J_2$. Therefore, it is enough to prove that $e(G) \leq 3$ for every sporadic simple  group $G$. 
If $g \in G$ is a product of two elements $x,y$ of order $3$ then
$$g \cdot g^x \cdot g^{x^2} = xy \cdot yx \cdot x^{-1}yx^2 = xy^3x^2 = 1.$$
By a result by Malcolm \cite[Theorem 4.1]{Malcolm21}, every element of $G$ is a product of at most $2$ elements of order $3$ unless $G$ is in the list $HS,Co_2,Co_3,Fi_{22},Fi_{23},BM$. So we may assume that $G$ is in this list. The same theorem of Malcolm implies that, if $G \not \in \{HS,Co_2,Fi_{22}\}$, then the elements that are not products of two elements of order $3$ have order $2$, so they are real, hence we may assume that $G\in\{HS$, $Co_2$, $Fi_{22}\}$. By \cite{at}, the orders of the non-real elements of these groups are $11,20$ for $HS$,   $14$, $15$, $23$, $30$ for $Co_2$ and  $11$, $16$, $18$, $22$ for $Fi_{22}$. It follows that all the ``width-three classes'' in \cite[Table 1]{Malcolm21} are real, so the result follows.
\end{proof}

\section{Some low-dimensional classical groups}

We first record the following general fact:

\begin{lemma}\label{u26}
Under the assumptions of Lemma {\rm \ref{u5n}}, suppose that $p>2$. Then $e_{C_G(s)}(u)\leq 3$, in particular,   $e_G(u)\leq 3$.
\end{lemma} 

\begin{proof}
The result follows from Lemma {\rm \ref{u5n}} and Corollary \ref{cor881}.
For the second conclusion of the lemma take $s=1$.  
\end{proof} 

\subsection{Groups $\PSU_3(q)$}

\begin{lemma}\label{u3q1} Let $G=\PSU_3(q)$,  $3|(q+1)$, $q> 2$. Then $e(G)=3$.  
\end{lemma}

\begin{proof} Let $g\in G$ and let $g=su$ be the Jodan decomposition of $g$. If $u=1$ or $s=1$ then $e(g)\leq 3$ by Lemmas \ref{un2} and \ref{u26}, respectively.  Suppose that $s\neq 1$ and $u\neq 1$. Then $s$ is not regular. Let $h\in H=\SU_3(q)$ be such that    $g= Z(H)h$. (Note that $|Z(H)|=3.$) One easily observes that the Jordan form of $h$ is 
$\begin{pmatrix}a&1&0\\ 0&a&0\\0 &0&a^{-2}\end{pmatrix}$.  It follows that $h$ lies in a conjugacy class $C_5^{(k)}$ in \cite[Table 2]{Si}, which provides the character table of $G$. 
(See \cite[Table 1.1, p.565]{Ge} for the corrections.) In particular, $|u|=p$,
where $p|q$ is a prime, and $|s|\neq 3$ divides $q+1$. Note that $g,g^{-1}$ are not conjugate so $e(g)>2$.

We show that $e(g)=3$. For this we use Lemma \ref{eq} for $k=3$, specifically, the fact that $e(g)=3$ if and only if
$\sum_{\chi \in \mbox{Irr}(G)} \chi(g)^3/\chi(1) \neq 0$.

For $\chi \in \mbox{Irr}(G)$ set $A_{\chi}=\{\psi \in \mbox{Irr}(G)\ |\ \psi(1)=\chi(1)\}$. Using the character table of $G$, we observe that the following holds. 

\begin{itemize}
\item If $\chi(1)=q^2-q$ then $|\chi(g)| \leq 1$ and $|A_{\chi}|=1$.
\item If $\chi(1)=q^2-q+1$ then $|\chi(g)| <2$ and $|A_{\chi}| = (q-2)/3$.
\item If $\chi(1)=q^3$ then $\chi(g)=0$.
\item If $\chi(1)=(q+1)^2(q-1)/3$ then $\chi(g)=0$.
\item If $\chi(1)=q(q^2-q+1)$ then $|\chi(g)|=1$ and $|A_{\chi}|=(q-2)/3$.
\item If $\chi(1)=q^3+1$ then $|\chi(g)|=1$ and $|A_{\chi}|=(q^2-q-2)/6$
see \cite[Table 2]{Si}.  
\item If $\chi(1)=(q-1)(q^2-q+1)/3$ then $|\chi(g)|=1$ and $|A_{\chi}|=3$.
\item If $\chi(1)=(q-1)(q^2-q+1)$ then $|\chi(g)| \leq 3$ and $|A_{\chi}| \leq (q-2)(q+1)/6$. This is the number of pairs $(u,v)$ with $1\leq u\leq (q+1)/3$ and $u<v<2(q+1)/3$ \cite[Table 3.1]{Ge}. 
\end{itemize}
In the above discussion, all the possible irreducible character degrees were mentioned. Therefore 
\begin{align*}
\left| \sum_{1 \neq \chi \in \mbox{Irr}(G)} \frac{\chi(g)^3}{\chi(1)} \right| \leq & \frac{1}{q^2-q} + \frac{8(q-2)}{3(q^2-q+1)} + \frac{(q-2)}{3q(q^2-q+1)} + \frac{q^2-q-2}{6(q^3+1)} \\
& + \frac{9}{(q-1)(q^2-q+1)} + \frac{27(q-2)(q+1)}{6(q-1)(q^2-q+1)}.
\end{align*}

This is strictly less than $1$ if $q \geq 8$. The case $q=5$ can be checked by \cite{gap}. 
\end{proof}

\subsection{Groups $\PSU_4(q)$}

\begin{lemma}\label{rr2} Let $G=\PGL_2(q)$ or $\PU_2(q)$. Then every semisimple element of $G$ is real. \end{lemma}

\begin{proof} 
Note that the groups $\PGL_2(q)$ and $\PU_2(q)$ are isomorphic (see \cite[Lemma 4.9]{DZ8}). So it suffices to deal with $\PGL_2(q)$. Let $H=\GL_2(q)$ and let $g\in H$. Let $V$ be the underlying space for $H$. Suppose first that $g$ is reducible on $V$. Then, under some basis of $V$, we can express $g=\diag(a,b)$.  Then  $g'=\diag(b,a)$ is conjugate to $g$  and $gg'$ is scalar. Whence the result in this case.
 
Suppose that $g$ is irreducible.  Then, by Schur's lemma,  the centralizer of $g$ in ${\rm Mat}_2(q)$ is a field,  which is isomorphic to $\mathbb{F}_{q^2}$. The nontrivial Galois automorphism $\gamma$ of $\mathbb{F}_{q^2}/\mathbb{F}_q$ sends $x\in \mathbb{F}_{q^2}$ to $x^q$. By the Skolem-Noether theorem, $\gamma$ is realized via an inner automorphism of $\GL_2(q)$, and hence $g^q = hgh^{-1}$ for some   $h \in \GL_2(q)$. Then $gg^q\in \mathbb{F}_q$ is scalar in $\GL_2(q)$. So again $g$ is real in $\PGL_2(q)$.\end{proof} 

Remark. For $H \in \{ \GL_2(q), \U_2(q)\}$   the element $h\in H$ such that $ghgh^{-1}$ is scalar can be chosen in $\SL_2(q)$, $\SU_2(q)$, respectively. This is because $h$ can be replaced by $ht$, where $t\in C_G(g)$ with $\det t=\det x^{-1}$.

\begin{lemma}\label{6ps} Let $ G=\PSU_4(q)$, q odd,  and let $s\in G$ be a semisimple element. Then $e(s)\leq 4$. In addition, $e(G)\leq 6$ and $e(\PSU_4(2))=e(\PSU_4(3))= 3$.\end{lemma}

\begin{proof} The result for $\PSU_4(2)$ and $\PSU_4(3)$ follows by \cite{gap}. Now we prove the first part. By Lemma \ref{un2},  $e(g)\leq 4$ if $g\in G$ is semisimple, and $e(g)\leq 3$ if $g$ is unipotent (Lemma  \ref{u26}).   
 Let $H=\SU_4(q)=\SU(V)$, let $h\in H$ be such that $g=hZ(H)$ and let $h=su$ be the Jordan decomposition of $g$. Let $C$ be the conjugacy class of $h$ in $H$.

  Let $J$ be the Jordan form of $u$ in $\GL_4(q^2)$. Then $J\in\{J_4,\diag(J_3,1),\diag(J_2,J_2)\}$. If $J=J_4$  then  $su=us$ implies $s$ is scalar, so $e(g)\leq 3$. In the following discussion we use Lemma \ref{centsem}.

If $J=\diag(J_3,1)$ then $h=\diag(\mu\cdot J_3,\mu^{-3})$.
In the notation of \cite{Si} (see also \cite{Ge, Or}), $J_3\in \PSU_3(q)$ belongs to one of  $d=(3,q+1)$ conjugacy classes 
$C_3^{(0,k)}$, where $0 \leq k < d$ (which glue in $\PU_3(q)$).
 So, let $D$ be any of these classes,  and let $D'$ be the conjugacy
class of $C_8^{(1)}$. Then, $D'$ consists of elements of order
$(q^2-q+1)/d$. Using the character table of $\PSU_3(q)$, one concludes that $D'\subset D^2$. So there  exists $x \in \PSU_3(q)$ such that
the product $t:=J_3\cdot xJ_3x^{-1}$ has order $(q^2- q +1)/d$. It follows that $C^2$ 
contains an element $y=\diag(\mu^2t,\mu^{-6})$. As $t$ is irreducible in $\U_3(q)$, $t$ does not  have an eigenvalue in $\mathbb{F}_{q^2}$, and hence $y$ is regular semisimple.  
By Lemma \ref{Gow}, some product of $3$ conjugates of $y$ is scalar. So $e(g)\leq 6$.

Suppose that $J=\diag(J_2,J_2)$. 
Then $W=(\Id-u)V$ is a totally isotropic space of dimension $2$. Let $s_1\in \GL_2(W)$
be the restriction of $s$ on $W$. Then  the action of $s$ on $V/W$ is $s_2:=\gamma({}^ts^{-1}_1)$, where $\gamma$ is the nontrivial Galois automorphism 
of $\mathbb{F}_{q^2}/\mathbb{F}_{q}$ extended to $\GL_2(q^2)$  (see for instance \cite[Lemma 4.8]{ez1}).
 By Lemma \ref{rr2}, the product of two conjugates of $s_1$ by an element of $\SL_2(q^2)$ is scalar in $\GL_2(q^2)$. So the  product $h_1$, say, of  two conjugates of $h$ by some element of $H$ is the product of a matrix $t:=\diag(\mu \cdot \Id_2,\mu^{-q} \cdot \Id_2)$ and  a unipotent matrix $u'$ of the form   $\begin{pmatrix} \Id_2&A\\ 0& \Id_2\end{pmatrix}$, where $A$ is a $(2\times 2)$-matrix over  $\mathbb{F}_{q^2}$. Moreover,
$t\in \langle h_1\rangle$ and hence $u':=t^{-1}h$ is unipotent and $u't=tu'$. Then either
 $u'=1$ or $t$ is scalar. In the latter case $e(g)\leq 3$. As $\diag(\mu^{-q} \cdot \Id_2,\mu \cdot \Id_2)$ is conjugate to $t$ by a matrix $\begin{pmatrix}0&\Id_2\\ \Id_2&0 \end{pmatrix}$, we have $e(g)\leq 4$.\end{proof}

\subsection{Groups $\PSp_4(q)$, $q \equiv 3 \pmod 4$}

\begin{lemma}\label{s33}
Let $H=\SL_2(q)$ with $q$ odd. If $u \in H$ is a nontrivial unipotent element, there exist three conjugates of $u$ in $H$ whose product is $-1$.
\end{lemma}

\begin{proof}
Note that $H$ has two conjugacy classes of nontrivial unipotent elements which glue in $\GL_2(q)$, so that a suitable element of $\GL_2(q)$ swaps the two $H$-classes of unipotent elements (see \cite[Lemma 5.2]{TZ04} for a more general fact). So it suffices to prove the result for one of the two classes. Consider the following elements of $H$:
$$g_1=\begin{pmatrix}0&1\\ -1&2\end{pmatrix}, \hspace{.1cm}
g_2=\begin{pmatrix}2&1\\ -1&0\end{pmatrix}, \hspace{.1cm}
g_3=\begin{pmatrix}1&0\\ -4&1\end{pmatrix}, \hspace{.1cm}
h=\begin{pmatrix} 0 & 1 \\ -1 & 0 \end{pmatrix}, \hspace{.1cm}
k=\begin{pmatrix} 2 & 1/2 \\ 0 & 1/2 \end{pmatrix}.$$
The characteristic polynomial of $g_1,g_2,g_3$ is $(x-1)^2$, so they are unipotent and $g_1 g_2 g_3=-1$. Moreover $h^{-1}g_1h=g_2$ and $k^{-1}g_1k=g_3$.
\end{proof}

\begin{prop}\label{s4q} Let $G=\PSp_4(q)$.    Then $e(G)= 3$.\end{prop}

\begin{proof}  The case with $q=3$ follows by GAP \cite{gap}. 
Some unipotent elements of $G$ are not real, so $e(G)\geq 3$. 

Set  $H = \Sp(4,q) = \Sp(V)$,  $h\in H$ and $g$ the projection of $h$ into $G$. If $h$ is unipotent then $e(h)\leq 3$ by Corollary \ref{cor881}, as $\Phi(h)$ lies in two conjugacy classes of $H$ by Lemma \ref{u5n}.  
If $h$ is semisimple then $h$ is real so $e(h)=2$. Let $h=su$ be the Jordan decomposition of $h$.  If $s\in Z(H)$ then $g$ is unipotent in $G$. So we assume that $s\notin Z(H)$, $u\neq 1$.

Suppose first that $s$  has an eigenvalue $\lambda\in\{ 1,-1\}$, and let $V_1$ be the $\lambda$-eigenspace of $s$ on $V$. It is well known that $V_1$ is non-degenerate, and hence $\dim V_1=2$ and $V=V_1\oplus V_1^\perp$. By replacing $s$ by $-s$ we can assume that $\lambda=1$. As $us=su$, we have $uV_i=V_i$ for $i=1,2$. Then we can write $u=\diag(u_1,u_2)$, where $u_i\in \Sp_2(q)$ is the restriction of $u$ to $V_i$. In addition, $s=\diag(\Id_2,s_2)$, where $s_2\in \Sp_2(q)$ is the restriction of $s$ to $V_2$, and $s_2\neq \Id_2$ (as $s\notin Z(G)$).

Let $X$ be the stabilizer of $V_1$ in $H$. Then $X=X_1X_2$, where $X_1=\diag(\SL_2(q), \Id_2)$ and $X_2=\diag(\Id_2, \SL_2(q))$.

(i) Suppose that $u_1=1$. Then $u_2\neq 1, s_2=-\Id_2$. By Lemma \ref{s33},
there are 3 conjugates of $u_2$ in $\Sp_2(q)$ whose product is $-\Id$. Then the  product of respective conjugates of $s_2u_2$ is $s_2^3 \cdot (-\Id)=\Id$, so $e(g)=3$. 

(ii) Suppose $u_1\neq 1$, $u_2= 1$. If $s_2=-\Id$ then we  replace $h$ by $-h$
and argue as in (i) to get $e(g)=3$. So we assume $s_2\neq -\Id$, 
and then $s_2$ is regular semisimple in $\SL_2(q)$. As $q>3$, the group $\PSL_2(q)$ is simple; by Lemma \ref{Gow},  
 there are 3 conjugates of $s_2$ whose product equals $ \Id$ or $-\Id$.
In the latter case, as $u_1\neq 1$, there are 3 conjugates of $u_1$ in $\SL_2(q)$
whose product equals $-\Id$ (Lemma \ref{s33}). So $e(g)=3$ in this case.
In the former case, as $q$ is odd, by Lemma \ref{u5n}, $\Phi(u_1)$ lies in at most two conjugacy classes of $\SL_2(q)$. 
Then the proof of Corollary \ref{cor881} shows that  $\langle u_1 \rangle$ contains 3 conjugates of $u_1$   whose product equals 1.  This again implies $e(g)\leq 3$. 
 
(iii) Suppose that  $u_1,u_2\neq 1$. Then $s_2=-\Id$. As $\SL_2(q)$ has two conjugacy classes of unipotent elements (and they are not real), we can assume that $u_2\in\{u_1,u_1^{-1}\}$, so  we can assume that  $h=\diag(u_1,-u_1)$ or $\diag(u_1,-u_1^{-1})$. Let $y=\begin{pmatrix}0&\Id\\-\Id&0\end{pmatrix}$. 
Then $y\in H$ and $h_2:=h^y=\diag(-u_1,u_1)$ or $\diag(-u_1^{-1}, u_1)$, respectively.  In the second case $h h^y=-\Id$, so $e(g)=2$.
  
So $h=\diag(u_1,-u_1)\subset X_1X_2$. We have seen above that there are 3 conjugates of  $u_1$ in $X_1\cong \SL_2(q)$ whose product is $\Id$, and there are three conjugates of $-u_1$ in $X_2\cong \SL_2(q)$ whose product is $\Id$. So $e(g)=3$. 
 
Next suppose that $s$ does not have eigenvalue 1 or $-1$. Let $S= \langle s \rangle$.  
If the irreducible  constituents of $V|_S$ are pairwise non-equivalent then $s$ is regular. It follows that either all these have dimension $1$ or they all have dimension $2$. Indeed, suppose the contrary; then one of them is of dimension $2$ and the others are of dimension $1$. So $s$ has eigenvalue $\lambda \in {\mathbb F}_q$. As $\lambda \neq \pm 1$ and $s$ is real,  $s$ has eigenvalue $\lambda^{-1}\neq \lambda$, and hence all three irreducible constituents of $V|_S$ are paiwise non-equivalent, a contradiction. 

So we have two cases:
(a) $V|_S$ is a sum of two equivalent irreducible constituents of dimension 2 and 
(b) $s$ has  two distinct eigenvalues $\lambda,\lambda^{-1}\in \mathbb{F}_q$, each of multiplicity $2$.

(a) In this case  $u\in C_{G}(s)\cong \U_2(q)$ (see for instance \cite[Lemma 6.6]{ez1}).  As the non-trivial unipotent elements in $\U_2(q)$ are conjugate \cite[Lemma 6.1]{TZ04}, and $s$ is real, we conclude that $g$ is real, hence $e(g)=2$ in this case. 
 
 (b) $V=W_1+W_2$, where $W_1,W_2$
is the  $\lambda$- and $\lambda^{-1}$-eigenspace of $s$ on $V$, respectively.   As $\lambda\neq \pm 1$, each of them is totally isotropic. 
Let $Y=\{x \in H: xW_i=W_i$ for $i=1,2\}$.
It is well known that $Y\cong \GL_2(q)$ (see, for instance, \cite[Lemma 4.8]{ez1} for a more general fact).  Since $us=su$, the eigenspaces of $s$ are $u$-stable, therefore $u\in Y$. As all  
non-trivial unipotent elements in $\GL_2(q)$ are conjugate,
and $s$ is real, we conclude that $g$ is real, hence $e(h)=2$ in this case.\end{proof}

\section{Classical groups}

\subsection{Symplectic groups}

The following result follows from Theorem \ref{un4}, Proposition  \ref{s4q}  and Lemma \ref{ee7}.

\begin{theor} \label{symplectic}
Let $G=\PSp(2n,q)$, $n>1$, be a simple symplectic group. If $q \not \equiv 3 \mod 4$, then $e(G)=2$. If $q \equiv 3 \mod 4$ then $e(G) \leq 6$ for $n>4$ and $e(PSp_4(q))=3$. 
\end{theor}
 
\subsection{Orthogonal groups}

Let $\eta:\SO^\pm_{2n}(q)\rightarrow \SO^\pm_{2n}(q)/\Omega^\pm_{2n}(q)$ the natural homomorphism.
We say that $g$ has spinor norm $-1$ if  $|\eta(g)|=2$ and 1 otherwise. 
Recall that $\Omega_2^\pm (q)$ is a cyclic group of order $(q\pm 1)/(2,q-1)$ \cite[Proposition  2.9.1(iii)]{KL}. 
It follows that the elements of order $q-(\pm 1) $ in $\SO_2^\pm (q)$ have spinor norm  $-1$ if $q$ is odd and 1 if $q$ is even.

\begin{lemma}\label{au8} Let $u$ be a unipotent element in $H=O^\pm _{2d}(q)=O(V)$ with $q$ even, $d$ odd. Then $C_H(u)$ is not a subgroup of $\Omega^\pm _{2d}(q)$.\end{lemma}

\begin{proof} By \cite[Lemma 4.9]{TZ04}, there exists an orthogonal decomposition $V=V_1 \oplus \ldots \oplus V_f$, such that each $V_i$ is nondegenerate, $uV_i=V_i$ for all $i$ and the Jordan block sizes of $u$ on $V_i$ are the same and distinct on $V_i,V_j$ for $i\neq j$, $i,j\in \{1,\ldots,f\}$. Let $u_i$ be the resriction of $u$ on $V_i$. It suffices to observe that $C_{O(V_i)}(u_i)$ is not contained in $\Omega (V_i)$ for some $i$. As $q$ is even, $\dim V_i$ is even for every $i$  by \cite[Proposition 2.5.1]{KL}, so  $\dim V_i\equiv 2\pmod 4$ for some $i  $. It suffices to deal with this particular $i$. 
To this end,   Let $a,m$ be such that   the Jordan blocks of $u_i$ consists of $a$ blocks of the same size $m$.
 
If $a$ is odd then then $u_i\notin \Omega (V_i)$ by  \cite[p.182]{C}. Suppose that $a$ is even. Then $m$ is odd and $u_i\in \Omega (V_i)$.  
 
We use \cite{LiS}. As  $m$ is odd, $u$ is of type $W(m)^{a/2}$ in the notation of \cite[p.117, Theorem 7.3(iii)]{LiS}. In this case $\Omega(V_i)\cong \Omega^+_{am}(q)$ and there is a single conjugacy class of elements of type $W(m)^{a/2}$ in  $\Omega (V_i)$. Therefore, it suffices to observe the following. Let $M$ be a maximal totally singular
subspace of $V_i$; then $V_i = M \oplus M'$, where $M'$  is a maximal totally singular
subspace of $V_i$ too. There is a basis $B$, say, of $V_i$ such that $B\cap M$,
$B\cap M'$ are bases of $M,M'$, respectively, and the Gram matrix of $B$ is 
$\begin{pmatrix}0&\Id_k\\ \Id_k&0\end{pmatrix}$
with $k=\dim M$. Let $u_i'=\diag(A,{}^t A^{-1})$, where $A\in \GL(M)$ is such that the Jordan forms of $u_i,u_i'$ coincide. Then $u_i'\in \Omega(V_i)$ is of type  $W(m)^{a/2}$. As there is a single conjugacy class of elements of this type in $\Omega(V_i)$, we can assume that $u_i=u'_i$. To show that 
$C_{O(V_i)}(u_i)$ contains an element  in $ O(V_i)\setminus \Omega(V_i)$, observe that 
if  $y=\begin{pmatrix}0&\Id_{k}\\ \Id_{k}&0\end{pmatrix}$
then $yu_iy^{-1}=\diag({}^tA^{-1},A)$. As $A$ and ${}^tA^{-1}$ have the same Jordan form, we have $xAx^{-1}={}^tA^{-1}$ for some $x\in \GL_{ k}(q)$, and $zyu_iy^{-1}z^{-1}=u_i $ for $z=\diag(x,{}^tx^{-1})\in \Omega(V_i)$. Note that $y$ is the product of $k$ reflections. As the spinor norm of a reflection is not equal to 1 and $k$ is odd,   the spinor norm of $y$ is not 1, and the result follows.\end{proof}

In the following proof we will use the fact that, if $g \in O(V)$ is semisimple and $W=\{v \in V\ |\ gv=v\}$, then $W$ is nondegenerate. To see this, set $X := W \cap W^{\perp}$. Then $X$, and hence $X^{\perp}$, is $g$-stable. Since $g$ is semisimple, $V=X^{\perp} \oplus Y$ for some $g$-stable subspace $Y$ of $V$. If $y \in Y$, $x \in X$ then $(gy-y,x) = (gy,x)-(y,x) = (gy,gx)-(y,x) = 0$, so $gy-y \in X^{\perp} \cap Y = \{0\}$. This means that $g$ acts trivially on $Y$, in other words $Y \leq W$. Since $W \leq X^{\perp}$, we deduce $Y=\{0\}$. This implies $X^{\perp}=V$, so $X=\{0\}$.

\begin{theor}\label{orth4}
Let $G=\Omega_n(q)$ for $n$ odd, or $\Omega^{\pm}_n(q)$ for $n>2$ even, and let 
$g\in G$. Then $e(g) \leq 4$, unless possibly $G=\Omega^{-}_n(q)$,  $n/2$ odd, $q$ even, $g$ is indecomposable, not semisimple and the composition factors of $g$ on $V=\mathbb{F}_q^n$ are isomorphic $\mathbb{F}_q\langle g\rangle$-modules of dimension $2$. In the exceptional case $e(G)\leq 8$. \end{theor}

\begin{proof} Suppose that $q$ is odd. If $g=xy$ with $x,y$ involutions, then $g \cdot g^x = 1$, so $e(g) \leq 2$. If $g=xyz$ with $x,y,z$ involutions, then $g \cdot g^z \cdot g^x \cdot g^{zx} = 1$, so $e(g) \leq 4$ (Lemma \ref{gr1}(3)). By \cite[Theorem 8.5]{kt1},   every element of $G$ is a product of at most $3$ involutions, so $e(G) \leq 4$, unless possibly $(n,q) \in \{(3,3),(4,3),(5,3)\}$. These three cases can be checked with \cite{gap}.

Assume that $q$ is even. Then $Z(G)=1$ and $G$ is a subgroup of index 2 in $O(V)$. If $n \geq 3$ is odd, then $\Omega_{n}(q) \cong \PSp_{n-1}(q)$ and the result follows from Theorem \ref{symplectic}. Now assume that $n$ is even. If $4|n$ then $e(G)=2$ by Theorem \ref{un4}, items (3),(4).

Assume that $q$ is even, $n \equiv 2 \mod 4$ (in particular $n\geq 6)$. Let  $G=\Omega^{\pm}_{n}(q)$ and let $V$ be the underlying orthogonal space for $G$.  Let $g=su\in G$,   where $u$ is unipotent, $s$ is semisimple and $su=us$. 

(i) Suppose that  $gW=W$, where $W$ is a non-degenerate proper subspace of $V$ such that $g$ acts trivially on $W^\perp$. Then $g$ is real.

Indeed, $W^\perp$ contains a non-degenerate subspace $W_1$, say, of dimension 2.
Then $4|\dim W_1^\perp$ and every element of $O(W_1^\perp)$ is real by  \cite{Gw}.
Therefore, $g^x=g^{-1}$ for some $x\in O(V)$ such that $x$ acts trivially on $W_1$.
So there exists $x_1\in G$ such that $x_1$ stabilizes $W_1^\perp$ and acts on it as $x$.  Whence the claim. 

(ii) Suppose that $V=V_1 \oplus V_2$, where $V_1, V_2$ are non-zero
 non-degenerate subspaces of $V$ and $gV_i=V_i$ for $i=1,2$. (Sometimes one says that $g$ is orthogonally decomposable.) Then $e(g)\leq 4$.

By reordering $V_1, V_2$ we can assume that $4|\dim V_2$. 
Let $g_1,g_2$ be the restrictions of $g$ to $V_1, V_2$, respectively. By \cite{Gw}, $xg_2x^{-1}=g_2^{-1}$ for some $x\in \Omega (V_2)$. Let $y\in G$ be such that $yV_i=V_i$ for $i=1,2$ and $y$ coincides with $x$ on $V_2$. Then $gygy^{-1}=\diag(g_1',\Id)$. By (i), this element is real, and hence $e(g)\leq 4$.

(iii) Suppose that $g$ is irreducible on $V$. Then $e(g)\leq 3$. In addition, there are 2  conjugates of $g$ whose product is a real and hence the product of some 4   conjugates of $g$  equals 1.

Note that $G$ is simple for $n>2$ even. So the former claim follows from Lemma \ref{Gow} as $g$ is regular semisimple in this case. Moreover, by \cite{Gow} every semisimple element $t\in G$ is a product of two conjugates of $g$. We can choose $t$ real. Indeed,
let $W$ be a non-degenerate subspace of $V$ of dimension 4.  
Then every element of $\Omega(W)$ is real (by the above), so we can choose 
for $t$ any semisimple element of $G$ acting  trivially on $W^\perp$. Whence the second claim.    

(iv) $e(g)\leq 4$ unless $s$ is homogeneous on $ V$. 

To prove this, suppose the contrary. Let $V_1,\ldots, V_m$ be the homogeneous components of $s$ on $V$, where $m>1$. Each homogeneous component is a largest direct sum of isomorphic $\mathbb{F}_q \langle s \rangle$-submodules. As $gs=sg$, we have $gV_i=V_i$ for $i=1,\ldots,m$.  By \cite[Lemma 3.3]{SZ02},  $V_1$ is either non-degenerate or totally singular. In the former case $e(g)\leq 4$ by (ii). 
Strictly speaking, \cite[Lemma 3.3]{SZ02} claims that $V_1$ is either non-degenerate or totally isotropic with respect to   the symplectic form  
associated with quadratic form defined $O(V)$. In the latter case this implies that  $V_1$ is totally singular.  (Indeed, otherwise the totally singular vectors of $V_1$ form a subspace $V_0$, say,  and every vector of $V_1\setminus V_0$ is not singular.  Let $v,v'\in V_1$ be non-singular vectors and $Q$ the quadratic form on $V$ defining $O( V)$. Replacing $v'$ by a multiple
we can assume   that $Q(v)=Q(v')$, and then $Q(v-v')=0$ so $v-v'\in V_0$. 
So $\dim V/ V_0=1$, and hence $s$ has eigenvalue 1. As $s$ is 
homogeneous on $V_1$, it follows that $s=1$, a contradiction.)

Let $m=2$. In this case let $H$ be the stabilizer of $V_1,V_2$ in $G$, so $g\in H$. Then $H\cong \GL_{n/2}(q)$. In fact, there is a basis in $V$ such that the matrix   of $h$ on $V_2$ is the inverse-transpose of that on $V_1$, which yields 
an isomorphism  $\GL_{n/2}(q)\rightarrow H$. (Note that the image is contained in $G$
as $|O(V):G|=2$ whereas the index  $|H:H'|$ is odd  for $n/2>2$ and $q$ even. The latter holds as $n\geq 6$.)

We claim that  $e(g)\leq 4$ if $g$ is decomposable as an element of $X:=\GL_{n/2}(q)$. Indeed, otherwise let $V_1=W_1\oplus W_2$, where $gW_i=W_i\neq 0$ for $i=1,2$. It is well known that there exists a \textit{Witt basis} $b_1,\ldots,b_n$ of $V$ with respect of $W_1$, which means that $b_1,\ldots,b_d$
is a basis of $W_1$ for $d=\dim W_1$, $b_{d+1},\ldots,b_{n-k} \in W_1^\perp$, and if 
$M$ is the matrix of the action of $g$ on $W_1$ with respect to the basis $\{b_1,\ldots,b_d\}$ then $W_3 = \langle b_{n-d+1},\ldots,b_n \rangle$ is $g$-stable and the matrix of the action of $g$ on this basis of $W_3$ is ${}^tM^{-1}$. Moreover, $W_1 \leq V_1 \leq V_1^{\perp} \leq W_1^{\perp}$, so the subspace $W_1+W_3$ is non-degenerate. As $W_2 \neq 0$ and $W_3 \cap W_1 = \{0\}$ we have $W_1+W_3 < V$. Since $V=W_1+W_3+(W_1+W_3)^\perp$ we have $e(g)\leq 4$ by (ii).

So we assume that $g$ is indecomposable as an element of $X=\GL_{n/2}(q)$. Let $h,t,v\in X$ be the elements of $X$ that are images of $g,s,u$ under the isomorphism $H \rightarrow X$. 
By Lemma \ref{in3}, $h$ is contained in a subgroup $X_1 X_2$, where $X_1\cong \GL_k(q)$, $X_2\cong \GL_l(q)$, $kl=n/2$, $[X_1,X_2]=1$ and $t\in X_1$ is irreducible as an element of $\GL_k(q)$, $v\in X_2$ is a Jordan block of size $l$ in $\GL_l(q)$. (Note that $X_1  X_2$ is isomorphic to the central product $\GL_k(q)\circ\GL_l(q)$, and this is the image of the Kronecker product homomorphism  $\GL_k(q)\times  \GL_l(q)\rightarrow GL_{n/2}(q)$.)

Note that $k$ is odd as so is $n/2$. Suppose that $k>1$. Note that there exists a non-scalar semisimple element $y_1\in \GL_{k}(q)$ having eigenvalue 1 such that $\det y_1 = \det t^2$. (Indeed, if $\det t^2 = a \neq 1$ then we can choose $y_1=\diag(a,\Id)$;  if $\det t^2 = 1$ and $q>2$ then we can choose  $y_1 = \diag(a,a^{-1},\Id)$ for $1\neq a \in \mathbb{F}_q$; if $q=2$ then we can choose $y _1$ of order 3.)
 By Lemma \ref{GuT1}, $y_1$ is a product of two conjugates of $t$. Using the isomorphism $X\rightarrow H$, we observe that
the product of the respective two conjugates of $s$ in  $H$ is a semisimple element $y$, say, that has eigenvalue $1$.  As $v$ is real in $X_2$, we conclude that  the product of some two conjugates of $g$ in $H$ is of the form $y':=\diag(y,y^{-1})$. Since $y$ is semisimple, the 1-eigenspace $W'$ of $y'$ on $V$ is non-zero and nondegenerate. By  (i) above,  $y'$ is real, so $e(g)\leq 4$.
 
Let $k=1$. Then $g$ stabilizes a singular line $\langle v\rangle $ on $V$.
Then $gv=av$ for some $a\in\mathbb{F}_q$. If $a= 1$ then the $1$-eigenspace of the semisimple part $s$ of $g$ on $V$ is non-zero and nondegenerate, and it is also $g$-invariant since $gs=sg$, so the result follows by (ii). If $a\neq 1$ then the $a$- and 

Then the result follows from (ii), unless $V=W+W'$. This means that $V$ is of type $+$,
which is a contradiction.  

(v) By (iv), we can assume that $m=1$, that is,  $s$ is homogeneous on $ V$. Let $l$ be the composition 
length of $s$ on  $V$, where $l>1$ (if $l=1$ then $s$ is irreducible on $V$ and $e(g)\leq 4$ by (iii)).

Note that $\langle s\rangle $ is completely reducible. It follows from  Schur's lemma  (applied to every irreducible constituent of $s$ on $V$) that $K$, the $\mathbb{F}_q$-span of $\langle s\rangle$ in $\mbox{End}_{\mathbb{F}_q}(V)$, is a field, so $V$ can be viewed as a vector space over $K$. Indeed, under some basis, $s$ is represented by a matrix $\diag(t,...,t)$, where $t$ is an irreducible matrix. Then the span of $s$ is isomorphic to the span of $t$.

It is shown in \cite[Lemma 6.6(1)]{ez1} that $|K|=q^{2k}$, where $2k=n/l$ is the dimension of every composition factor, and $C_{O(V)}(s)\cong \U_l(q^k)$. 
(The meaning of $k,l$ here differ from those in (iv).) As $n/2$ is odd, both $k,l$ are odd. Let $V_K$ denote $V$ viewed as a $K$-module.  Then, by \cite[Lemma 6.6(2)]{ez1}, $V_K$ is a unitary $K$-space, and the unitary geometry on $V_K$ agrees with orthogonal geometry on $V$ in the sense that for a $K$-subspace $W$ the space $W^\perp$ is the same in each geometry. In addition, $W$ is totally singular (non-degenerate) if and only if $W$ is totally isotropic (respectively, non-degenerate) in the unitary geometry.
In addition, $g$ stabilizes no non-degenerate   subspace of $V$ if and only if
$g$ stabilizes no non-degenerate subspace of $V_K$  equivalently,
$u$  stablizes no non-degenerate  subspace of $V_K$. By \cite[Lemma 6.1]{TZ04},  this is equivalent to saying that the Jordan form of $u$ on $V_K$ consists of a single block. 

It is well known that $V_K$ is a direct sum of non-degenerate one-dimensional subspaces.
It follows that $V$ decomposes as a direct sum of  non-degenerate $s$-stable subspaces, each  of dimension $2k$. Therefore, $s$ is contained in the stabilizer of such a decomposition and hence we can write $s=\diag(t,\ldots,t)$, where $t\in O(W)$ and $W$ denotes one of the summands. As $|s|$ is odd, $t\in \Omega(W)$. In addition, as $s$ is irreducible in $O(W)$,
we have $O(W)\cong O_{2k}^-(q)$ see \cite[Satz 3(b)]{Huppert}. This means that 
$G=\Omega^-_{2n}(V)$ (as $V$ is a direct sum of $l$  non-degenerate subspaces of $-$ type and $l$ is odd, see \cite[Proposition 2.5.11]{KL}).  This implies that $e(g)\leq 4$ if $G=\Omega^+_{n}(q)$, or if $G=\Omega^-_{n}(q)$ and  $g$ is semisimple.

Let $S=\{\diag(a,\ldots,a)\ |\ a\in \Omega(W)\}$. Then $S\subset G$, $s\in S$ (as $|s|$ is odd). Let $\eta:S\rightarrow \Omega(W)$ be an isomorphism. 

Let $M'$ be a maximal totally isotropic subspace of $V_K$ and let $M$ be the corresponding subspace of $V$. Then $M$ is totally singular, $\dim M'=(l-1)/2$ and $\dim M=2k(l-1)/2$. Then there exist totally isotropic subspaces $M_1'<\cdots<M'_{(l-1)/2}$ with factors of dimension 1. Let $0=M_0<M_1<\cdots<M_{(l-1)/2}$ be the respective chain of $K$-stable subspaces of  $V$.
Then the stabilizer of all these subspaces is a parabolic subgroup $P$ of $G$ and $O_2(P)$
consists of all $x\in P$ acting trivially on the subsequent factors of the chain as well as 
on $M_{(l-1)/2}^\perp/M_{(l-1)/2}$. Let $Q_i$ be any complement of $M_i$ in $M_{i+1}$
for $i=0,\ldots,(l-1)/2$ (which are totally singular), and $Q_{(l-1)/2}$ is a complement of  $M_{(l-1)/2}$ in $M^\perp_{(l-1)/2}$, which is non-degenerate.  Note that $Q_{(l-1)/2}$ is non-degenerate, and $Q_i$ is totally singular for $i=0,\ldots,(l-1)/2$.
In addition, $P=O_2(P).L$, where
 $L=\{y\in P: yQ_i=Q_i\}$ for $i=0,\ldots,(l-1)/2$ is a Levi subgroup of $P$. Then the restriction of 
$L$ on $Q_i$ is $\GL(Q_i)$ for $i=0,\ldots,(l-3)/2$ and $\Omega^-_{2k}(q)$ for $i=(l-1)/2$.
 More precisely,  
$L\cong \GL(Q_1)\times \ldots \times \GL(Q_{(q-1)/2}) \times \Omega^-(Q_{(l-1)/2})$. 

Due to the Schur-Zassenhaus theorem we can assume that $s\in L$. Let $s_i$ be the projection of $s$ to $GL(Q_i)$  for $i=0,\ldots,(l-1)/2$. As $s$ is homogeneous on $V$,
under some basis $B$ of $M_{(i-1)/2}$ the matrix of $s_i$ on $Q_i$ is   the same for $i=1,\ldots ,(i-1)/2$. We denote by $S$ the subgroup of $L$ consisting of the elements which have the same matrix on each $Q_i$ for $i=1,...,(i-1)/2$ under the above basis $B$. As $L$ acts on $Q_{(l-1)/2}$ as $\Omega^-_{2k}(q)$, we observe that $S\cong \Omega^-_{2k}(q)$ and acts faithfully on every composition factor of $g$ on $V$.  
(We can choose  a Witt basis $B$ of $V$ such that $B\cap Q_i$ is a basis of $Q_i$ for $i=0,\ldots,(k-1)/2$; then $S$ can be written as the set of matrices in $L$ of the form $\diag(t,\ldots,t)$, $t\in \Omega^-_{2k}(q)$.) 

Suppose that $k>1$. As $k$ is odd,  $k\geq 3$ and  $\Omega^-_{2k}(q)$ is simple. 
By \cite{Gow},
every semisimple element $t\in S\cong \Omega^-_{2k}(q)$ is a product of two conjugates of $s$.  We can choose $t$ such that $m(t)$, the eigenvalue 1 multiplicity  of $\eta(t)$, is not a multiple of 4 (indeed, as $2k\geq 6$, we can find $\eta(t)$ of order 3  with $m(\eta(t))=2k-4$). Then the  eigenvalue $1$ multiplicity of $t$  is $2d:=l \cdot m(\eta(t))$, which is not a multiple of $4$ too (since $l$ is odd). Let $x\in S$ be such that $sxsx^{-1}=t$.  Then $h := gxgx^{-1} = us xsx^{-1} \cdot xux^{-1} = utxux^{-1} = tu'$, where $u'=t^{-1}ut \cdot xux^{-1}\in O_2(P)$. Let $V_0$ be the $1$-eigenspace of $t$ on $V$; then $V_0$ is non-degenerate and  $hV_0=V_0$. 
By the above, $\dim V_0\equiv 2\pmod 4$. Then $4|\dim V_0^\perp$.  
Let $h=\diag(h_0,h_1)$ with $h_0,h_1$ the restrictions of $h$ to $V_0$, $V_0^\perp$, respectively.  
Then $h_0 \in O(V_0)$ is unipotent and $h_1$ is real in  $O( V_0^\perp)$. Then $h_0$ is real in $O(V_0)$ and $h_1$ is real in $O (V_0^\perp)$, so $h$ is real in $O(V)$. As $(\dim V_0)/2$ is odd, $C_{O(V_0)}(h_0)$ contains an element that does not lie in $\Omega(V_0)$.
Since $|O(V)/\Omega(V)|=2$, it follows that $h$ is real in $G$, and hence $e(g)\leq 4$.

Finally suppose that $k=1$. By \cite[Theorem 1.8]{TZ04}, $u$ is real in $\SU_l(q^k)$, and $e(s^2)\leq 4$. So $e(g)\leq 8$.
\end{proof}

Next we look at the lower bound for $e(G)$ with $G$ as above. 
We need the following lemma.

\begin{lemma}\label{qe1} Let $F$ be an algebraically closed field and let  $A,B\in \GL_n(F)=\GL(V)$ be diagonalizable elements with exactly two eigenvalues. Let $X=\langle A,B\rangle$. Then the composition factors of X on V are of dimension $1$ or $2$. \end{lemma}

\begin{proof} Suppose the contrary, and let $W$ be a composition factor of dimension $d>2$. Since $A$ and $B$ are diagonalizable, they both have exactly two eigenvalues on $W$. So we can assume that $W=V$, so $X$ is irreducible.  

Let $\lambda_1,\lambda_2$,  $\mu_1,\mu_2$ be the eigenvalues   of $A,B$, respectively, and we can assume  (by reordering them)  that 
the multiplicity of $\lambda_1$  and of $\mu_1$ is at least $n/2$.    Let $V_1,V_2$ be 
$\lambda_1,\mu_1$-eigenspaces of $A,B$, respectively. As $X$ is irreducible, we have $V_1\cap V_2=0$, and hence $\dim V_1=\dim V_2=n/2$ and $V_1+V_2=V$. Choose a basis $K$
of $V$ such that $K_i:=K\cap V_i$ is a basis of $V_i$, $i=1,2$, ordered so that the elements of $K_1$ are prior to those of $K_2$. Under such a basis the matrices of $A,B$ are of the form
$$\begin{pmatrix}\lambda_1 \cdot \Id&C\\ 0&\lambda_2 \cdot \Id\end{pmatrix}, \hspace{1cm}
\begin{pmatrix}\mu_2 \cdot \Id&0\\ D&\mu_1 \cdot \Id\end{pmatrix}$$ 
where $C,D$ are square matrices of size $n/2$, and $D\neq 0$ as $X$ is irreducible.
Let $\eta:V_1\rightarrow  V_2$ be the mapping defined by $\eta(v) := (B-\mu_2\Id)v\,$ 
for $v\in V_1$.
Since ${\rm ker}\, \eta$ is $X$-stable, $D$ is nonsingular, so $\eta(B_1)$ is a basis of $V_2$, hence we can choose $K_2=\eta(K_1)$ so that $D$ is replaced by the identity matrix.  
Now we can take $C$ to the upper triangular form  by
conjugating both $A,B$ by a suitable matrix $N=\diag(M,M)$ with $M\in \GL_{n/2}(F)$. Note that $NBN^{-1}=B$. Then the vectors $v_1=(1,0,\ldots, 0)$, $v_2=(0,\ldots, 0,1,0,\ldots ,0)\in F^{2n}$ span a 2-dimensional $X$-stable subspace of $V$ (here $1$ is located at the $(n+1)$-th position). \end{proof}

Consider the  matrix $J=\begin{pmatrix}0&\Id_n\\ \Id_n&0\end{pmatrix}$. Then $J$ is the Gram matrix of the natural basis $B$ of $V=F^{(n)}$. The orthogonal group $O^+_{2n}(F)$ is contained in $H=\{ h\in \GL_{2n}(F): JxJ^{-1}={}^tx^{-1}\}$, where ${}^t$ denotes the transposition operation. If $F$ is not of characteristic $2$ then $H\cong O^+_{2n}(F)$. If 
 $F$ is of characteristic $2$ then $H\cong Sp_{2n}(F)$ and $H$ contains subgroups isomorphic to  $O^\pm_{2n}(q)$ defined in terms of certain quadratic forms. The group $G$
in Theorem \ref{th2} is a subgroup of  $H$. 

We use this notation in the proof of the following result.

\begin{theor}\label{th2} Let $F$ be a field of characteristic $p\geq 0$.  For $n$ odd let $G=\SO^+_{2n}(F)$ if $p\neq 2$, and $G=\Omega^+_{2n}(F)$ if $p=2$. 
Suppose that $F^\times $ has an element $\mu$ such that $\mu^6\neq 1$.
Let $g = \diag(\mu\cdot \Id_n,\mu^{-1}\cdot \Id_n) \in G$. Then $e(g)>3$. In addition, if $\mu^{6}\neq \pm 1$ and $g'$ is the image of $g$ in $G/Z(G)$ then $e_{G/Z(G)}(g')>3$.
\end{theor}

\begin{proof} Let $V=F^{n}$ be the underlying space of $G$. Let $W,W'$ be the $\mu$-, $\mu^{-1}$-eigenspaces of $g$, respectively. Observe that $W,W'$ are totally singular (as a $\mu$-eigenvector is singular unless $\mu=\pm  1$) and $JW=W'$. 

We claim that $J\in O^+_{2n}(F)$. This is clear for $p\neq 2$.  Suppose  that $p=2$, and let $Q$ be the quadratic form on $V$ defining $O^+_{2n}(F)$. We need to show that $Q(Jv)=Q(v)$ for every $v\in V$. Recall that $Q(v+v')=Q(v)+Q(v')+(v,v')$, where
$(v,v')$ is a bilinear form defining $\Sp_{2n}(F)$. As $J\in \Sp_{2n}(F)$, it suffices to show that $Q(Jv)=Q(v)$ for  $v$ running over a basis of $V$. Let $B$ be the natural bsis of $F^n$.  Then $B\cap W$, $B\cap W'$ are bases of $W,W'$, respectively. Then $Q(v)=0$ for every $v\in B$ as both $W,W'$ are totally singular.  Since $Jv\in W\cup W'$ for $v\in B$, the claim follows.  

In addition, $J\notin G$. This is obvious for $p\neq 2$ as $\det J=-1$. If $p=2$ then $J$ is a product of $n$ reflections and $n$ is odd, so the claim follows from the definition of the spinor norm, see \cite[p. 29]{KL}. 

We claim that $g$ is not real in $G$ for $p\neq 2$.   Indeed, if $xgx^{-1}=g^{-1}$ for $x\in G$ then $Jxgx^{-1} J^{-1}={^t}g=g$, so $Jx\in C_{\GL_{2n}(F)}(g)$, hence $Jx=\diag(y,y')$ for some $y,y'\in \GL_n(F)$. Since $Jx \in O_{2n}^{+}(F)$, we have $y'={^t}y^{-1}$ hence $\det Jx = 1$. This is a contradiction, as $\det J=-1$ (since $n$ is odd) and $\det x = 1$. (This argument is valid if $\mu^2\neq 1$.) 
Therefore,  $e(g)>2$. 

Observe that $g' = gZ(G)$ is not real in $G/Z(G)$ if $\mu^{4}\neq 1$. Indeed, if $xgx^{-1}= z g^{-1}$ for $1\neq z\in Z(G)$ and $x \in G$ then $z^2=1$ and $xg^2x^{-1} = g^{-2} = \diag(\mu^{-2}\cdot\Id,\mu^2\cdot\Id)$. This is false unless $\mu^4=1$ as we have just seen. So $e_{G/Z(G)}(g')>2$.

If $p=2$ then $g$ is not real in $O(V)$. Indeed,    if $xgx^{-1}=g^{-1}$ for $x\in G$ then 
the spinor norm of $x$ and of $Jx=\diag(y,{}^ty^{-1})$
is 1, and hence so is that of  $J$, which is false. So  $x\notin \Omega^+_{2n}(F)$.

Thus $g$ is not real in $G$.
Assume by contradiction that $e(g) \leq 3$. Then $e(g)=3$, so there exists $h \in G$ such that $g,h$ are conjugate in $G$ and $gh$ is conjugate to $g^{-1}$. Set $X := \langle g,h \rangle$. Then 

(*)  $g,h$ have no common eigenvector on every composition factor $M$ of $X$ on $V$.

Indeed, if $0\neq v\in M$ and $gv=\mu v$, $hv=\mu^{\pm 1}v$ then $gh v=
\mu \cdot \mu^{\pm 1} v \in\{v, \mu^2 v\}$. As $\mu^{\pm 1}$ are the only eigenvalues of $g^{-1}$ we have $\mu^2=\mu^{-1}$, whence $\mu^3=1$. 

This implies, by Lemma \ref{qe1}, that all composition factors of $X $ on $V$ are of dimension 2.

(**) Observe that $O_2^\pm  (F)$ has an abelian normal subgroup  
of index 2 (see for instance \cite[Proposition  2.9.1(iii)]{KL}). If $p\neq 2$ then $\SO_2(F)$
is abelian,  see \cite[\S 6, item (3) on p. 49]{Di}.

Let $W$ be a minimal non-trivial $FX$-submodule of $V$. By (**),  $W$ is degenerate and hence totally singular.  
(Indeed, otherwise the restriction $X_W$ of $X$ to $W$ is a subgroup of $O(W) $. 
If $p=2$ then all semisimple elements of $O(W)$ commute and hence $g,h$ have a common eigenvector on $W$.   Let $p\neq 2$. As both $\mu,\mu^{-1}$ are eigenvalues of 
$g$ (and of $h$), we observe that $X_W\subset \SO(W)$. So $X_W$ is abelian and hence cannot be irreducible.) 
 
Let $L$ be a maximal totally singular $X$-stable subspace of $V$.
Then $L^\perp/L$ is a non-degenerate orthogonal space and, by the maximality of $L$,
this has no totally singular  $X$-stable subspace.
However, if $W_1$ is a minimal $FX$-submodule of $L^\perp/L$ then, by (*) as above,  we conclude that $W_1$ is totally singular. This is a contradiction unless $L=L^\perp$, and then $\dim V=2\dim L$. Then $\dim L$ is even as all composition factors of $X$ on $L$ are of dimension 2. Therefore, $4|\dim V$, a contradiction.

 If $p=2$  then $Z(G)=1$ and the additional statement is trivial. 
If $p\neq 2$ then $|Z(G)|=2$. We have seen above that $g'$ is not real if $\mu^6 \neq \pm 1$, so   $e(g')\geq 3$. Suppose that $e(g')= 3$. Then (*) remains true if $\mu^6\neq 1$, and all further considerations are valid with no changes, leading to a contradiction.
\end{proof}

\begin{lemma}\label{dd2} Let V be an orthogonal space over ${\mathbb F}_q$ of dimension $2n$. 

$(1)$ If $V$ is of type $+$ then $V$ is an orthogonal sum of $n$ dimension $2$ non-degenerate subspaces of type $+$, and if n is even then $V$ also is an orthogonal sum of $n$ dimension $2$ non-degenerate subspaces of type $-$. 

$(2)$ If $V$ is of type $-$ and $n$ is odd then $V$ is an orthogonal sum of $n$ dimension $2$ non-degenerate subspaces of type $-$.

$(3)$ Suppose that n,q are odd. Then $Z(\Omega(V))\neq 1$ 
if and only if $V$ is of type $+$ and   $4|(q-1)$  or $V$ is of type $-$ and $4|(q+1)$.  
\end{lemma}

\begin{proof}
This follows from \cite[Propositions 2.5.11(ii) and 2.5.13]{KL}.\end{proof}

\begin{corol}\label{oo7} 
$(1)$ Let $G=\Omega^-_{2n}(q) $, n odd.  
If $q\neq 2,3,5,11$ then $e(G)>3$; if $q\neq 2,3,5,7,11,23$ then $e(G/Z(G))>3$.  

$(2)$ Let $G=\Omega^+_{2n}(q)$, n odd. If $q\neq 2,3,4,5,7,13$. 
Then $e(G)>3$;  if $q\neq 2,3,4,5,7,9,13,25$ then $e(G/Z(G))>3$.\end{corol} 

\begin{proof} Let $V$ be the underlying orthogonal space for $G$. Set $H=\SO(V)$.

(1) By Lemma \ref{dd2}(1), $V$ is a direct  sum of 2-dimensional non-degenerate subspaces of type $-$ orthogonal to each other. Therefore, the stabilizer of this decomposition in $ H$ contains the direct product of  $n$ copies of $\SO_2^-(q)$. As  $\SO_2^-(q)$ has an element $x$, say, of order $q+1$ with eigenvalues $\nu,\nu^{-1}$ over $\overline{{\mathbb F}}_q$, where $\nu$ is a primitive $(q+1)$-th root of unity, we can construct $h\in H:=O(V)$ as $\diag(x,\ldots, x)$. Note that $h$ is diagonalizable over an algebraic closure of $\mathbb{F}_q$ with eigenvalues $\nu,\nu^{-1}$ of multiplicity $n$. 

If $W\subset V$ is a non-degenerate subspace then the subgroup $Y_W$ of elements $y\in O(V)$ such that $yW=W $, $yW^\perp=W^\perp$ and $y$ acts  trivially on $W^\perp$ is isomorphic to $O(W)$ and $Y_W\cap \Omega(V)\cong \Omega (W) $.  Applying this observation to every 2-dimensional 
space of dimension 2 and of $-$ type from the above decomposition of $V$, we observe (as $O(V)/\Omega(V)$ is of exponent 2 and $n$ is odd that $h\in \Omega(V)$ if and only if $x\in \Omega_2(q)$. If $q$ is even then $x\in \Omega^-_2(q)$ and hence $h\in \Omega(V)$. If $q$ is odd then $x\notin \Omega^-_2(q)$, $x^2\in \Omega^-_2(q)$, so $h\notin \Omega(V)$, $h^2\in \Omega(V)$.

 We wish to use
  Theorem \ref{th2}; note that if $e(g)>3$ over an algebraic closurce  of $\mathbb{F}_q$, then $e(g)>3$ in $G$ as well. If $q\neq 2,3,5,11$
then $(\nu^2)^6\neq 1$, so $e(g)\geq 3$ by  Theorem \ref{th2}. 
  If $Z(G)=1$ then this holds for $P\Omega^-_{2n}(q)$.  Suppose that $Z(G)\neq 1$. Then $4|(q+1)$ by Lemma \ref{dd2}. To use Theorem \ref{th2}, we now need  $(\nu^2)^6\neq -1$, or $\nu^{24}\neq 1$, equivalently $q \not\in \{3,5,7,11,23\}$. So $e(G/Z(G)) > 3$ in this case. If  $q$ is even then $|x|$ is odd and hence $a=1$, $h\in G$, and we set $g=h$. Then $\nu^{12} \neq 1$ if and only if  $q\neq 2$.

(2) In this case $V$ is a direct sum of 2-dimensional non-degenerate spaces of type $+$
orthogonal to each other.
Therefore, the stabilizer of this decomposition in $\SO^+_{2n}(q)$ contains the direct product of  $n$ copies of $\SO_2^+(q)$. Now $\SO_2^+(q)$ has an element $x$ of order  $q-1$.    As above, 
$h=\diag(x,\ldots, x)\in G $   for $q$ even. If $q$ is odd  then $h\notin G$, and   $g=h^2\in G$. If $\nu^{12} \neq 1$, equivalently $q \notin \{3,5,7,13\}$, then, as above, $e(G)> 3$ by  Theorem \ref{th2}.  If $Z(G)\neq 1$ then $4|(q-1)$ and we need 
$\nu^{24}\neq 1$,  equivalently $q \not\in \{3,5,7,9,13,25\}$. So $e(G/Z(G)) > 3$ in this case. If $q$ is even then $|x|$ is odd and hence $a=1$, $h\in G$, and we set $g=h$. Now $\nu^{12} \neq 1$ if and only if $q\notin\{2,4\}$. So again the result follows from Theorem \ref{th2}.
\end{proof}

\subsection{Groups $\PSL_n(q)$ }  

\begin{lemma}\label{lemmaPSL}
Let $n \geq 2$ be an integer, $F$ a field, $H=\SL_n(F)$, $G=\PSL_n(F)$. Let $\mu \in F^{\times}$ and let $g=\diag(\mu^{1-n},\mu,\ldots,\mu) \in H$. Assume there exists $l \in \{1,\ldots,n-1\}$ such that $g^{x_1} \ldots g^{x_l} \in Z(H)$ 
for some $x_1,\ldots,x_l \in \GL_n(F)$. Then $\mu^{ln}=1$. In particular, if $F$ is a finite field of  $q$ elements and $d=(n,q-1)$ 
 then $e(G) \geq e(g)\geq (q-1)/d$.
\end{lemma}

\begin{proof}
The last claim follows by choosing $\mu$ of multiplicative order $q-1$, since then $\mu^{ln}=1$ implies that $q-1$ divides $ln$, so $(q-1)/d$ divides $l$. So we prove the first claim, i.e. that $\mu^{ln}=1$. This is trivial if  $\mu^n=1$. Suppose that $\mu^n \neq 1$. Then $\mu^{1-n} \neq \mu$ so the matrix $g$ is not scalar, and the $\mu$-eigenspace of $g$ 
(and any its conjugate) is of dimension $n-1$. Suppose that  $g^{x_1} \ldots g^{x_l} = \lambda \cdot \Id \in H$ for some  $x_1,\ldots ,x_l \in H$. Let $W_i$  be the  $\mu$-eigenspace of $g^{x_i}$ for $i=1,\ldots ,l$ and let $W:=\bigcap_{i=1}^l W_i$. 
As $\dim W_i=n-1$ for all $i=1,\ldots,l$, we have  $\dim W \geq n-l$. Moreover $g^{x_1} \ldots g^{x_l}$ acts as the multiplication by $\mu^l$ on $W$, so $\mu^l=\lambda$ and hence $\mu^{ln} = \lambda^n=\det(\lambda\cdot \Id)=1$.
\end{proof}

If $d=q-1$, that is, $n$ is a multiple of $q-1$, then Lemma \ref{lemmaPSL}
reduces to the trivial claim that $e(G)>1$. So we can try to improve the result for this and related cases. For $q=2$ we have:

\begin{theor}\label{gg6}
Let $G=\GL_n(2)$, $n>1$.  Then $e(G)\leq 6$ and $e(s)\leq 3$ for $s\in G$ semisimple.  
\end{theor}

\begin{proof} 
Write $G=\GL(V)$. Suppose that $s\in G$ is semisimple. Then $V=\bigoplus_{i=1}^k V_i$, where $V_1,\ldots, V_k$
are irreducible $\mathbb{F}_2\langle s\rangle $-modules. Therefore, $s$ belongs to 
$$S = \{x \in G\ |\ xV_i=V_i\ \mbox{for}\ i=1,\ldots,k\} \cong \GL(V_1)\times \cdots\times \GL(V_k).$$ 
Let $s_i$ be the projection of $s$ into $\GL(V_i)$ for $i=1,\ldots,k$. 
Then $|C_{\GL(V_i)}(s_i)|$ is odd (by Schur's lemma), and hence $s_i$ is regular in $\GL(V_i)$. 
 If $\dim V_i>2$  then the group  $\GL(V_i)$ is simple, and then,  by  Lemma \ref{Gow}, there are $x_i,y_i\in \GL(V_i)$ such that $s_i x_is_ix_i^{-1} y_is_iy_i^{-1}=\Id$. This is also true if 
 $\dim V_i \leq 2$.  (This is trivial if $\dim V_i=1$.If $\dim V_i=2$ then $\GL(V_i)\cong S_3$, so $|s_i|\in 1,3$, and hence $s_i^3=1$.   
It follows that   $sxsx^{-1} ysy^{-1} =\Id$ for some $x,y\in S$. 

It is well known that $C_G(s)$ is a direct product of groups $\GL_{n_j}(2^{k_j})$ for various
$n_j,k_j$, see \cite[Lemma 3.2]{TZ04}. Let $g=su$, where $s\in G$ is semisimple,  $u\in G$ is unipotent and $us=su$. Then $u=\Pi_ju_j$, where  $  u_j$ is a unipotent element in $ \GL_{n_j}(2^{k_j})$. By \cite[Theorem 1.8(ii)]{TZ04}, $u_j$ is real in $\GL_{n_i}(2^j)$. Therefore,  $u$ is real in  $C_G(s)$. 
Now we use Lemma  \ref{xy8} with $x=s$, $y=u$ and $m=e_{C_G(s)}(u)=2$. As $s^2$ is semisimple, we have $e_G(s^2)\leq 3$, and hence $e_G(su)\leq 6$.
\end{proof}

\begin{lemma}\label{s25} $e(\GL_n(2))=3$ for $n=3,4,5$. 
\end{lemma}

\begin{proof} Observe that $e(\GL_n(2))=3$ for $n=3,4$ as $\GL_3(2)\cong \PSL_2(7)$ and $\GL_4(2)=A_8$. Let $n=5$.

Let $g\in G$ and let $g=su$, where $s\in \langle g\rangle$  is of odd order, $u$ is a 2-power.   If 
$s=1$ then $e(u)=2$. Suppose that $s\neq 1 $. In view of Lemma \ref{Gow} we can assume that $s$ is not regular. Then $s$ is reducible by Schur's lemma. Let $V_1,\ldots,V_k$ be the homogeneous components of $s$ on $V$, and we can assume that $\dim V_1\geq \dim V_2\geq \ldots \geq \dim V_k$. If $\dim V_k=1$ then $g\in \diag(\GL_4(2),1)$ and the result follows since we know that $e(\GL_4(2))=3$. So we  assume that $\dim V_k\geq 2$, and hence $k=2$, $\dim V_1=3$, $\dim V_2=2 $. Note that the homogeneous components of $s$ are $g$-stable. So $gV_i=V_i$ for $i=1,2$. 
Then $g\in  X_1,X_2$, $X_1\cong  \GL_3(2)$, $X_2\cong \GL_2(2)$. We can write  $g=\diag(g_1,g_2)$, and let $g_i=s_iu_i$ for $i=1,2$. Then $s_i$ is either irreducible on $V_i$ or $s_i=\Id$. If both $s_1,s_2$ are irreducible then $s$ is regular by \cite{at}.
So $s\neq 1,u\neq 1$, and either $g_1=s_1$, $g_2=u_2$ or  $g_1=u_1$, $g_2=s_2$.
In the latter case each $u_1$, $s_2$ is real, so $g$ is real. In the former case $|s_1|=7$, $|u_2|=2$, so $|g|=14$.

For this $g$ we can use the character table \cite[Page 70]{at}. Note that the 4 characters 
of degrees 496 take value $b_7,\overline{b_7},-b_7,-\overline{b_7}$ so their cubes sum to 0.  We have
$$\sum_{\chi\in{\rm Irr}(G)}\frac{\chi(g)^3}{\chi(1)}=1+\frac{-1}{155}+ \frac{1}{930}+\frac{1}{960} + \frac{-1}{1240}>0.$$
This proves the claim.
\end{proof}

\begin{lemma}\label{id4}
If  $g \in \GL_n(3)$ is indecomposable, then there exist $x,y \in \SL_n(3)$ such that $g \cdot g^x \cdot g^y$ is scalar.
\end{lemma}

\begin{proof} Let $g=su$ be the Jordan decomposition of $g$.
In the notation of Lemma \ref{in3}, $s\in M_1$, $u\in M_2$, and $s$ is irreducible in $M_1\cong \GL_k(3)$.  Let $M_3\subset M_1$, $M_4\subset M_2$ be subgroups isomorphic to $\SL_k(3), \SL_l(3)$, respectively. Since the ground field has $3$ elements, $ \langle s,M_3 \rangle$ is either $M_1$ or $M_3$. 
By Corollary \ref{GuT}, a product of $3$ conjugates of $s$ in $\langle s,M_3\rangle$ is scalar and the conjugation elements can be chosen in $M_3$. In addition, if $u$ is not real in $M_4$ then a product of 3 conjugates of $u$ in $M_4$ equals 1 by Lemma \ref{u26}. If $u$ is real 
then  $u$ is rational in $M_4$  by \cite[Lemma 2.2(ii)]{TZ05}. 
 Then $u$ is conjugate to $u^{-2}$,  so some product of $3$ conjugates of $u$ in $M_4$ equals 1. It follows that a product of $3$ conjugates of $g$ is scalar and the conjugation elements can be taken in $\SL_n(3)$.
\end{proof}

\begin{lemma}\label{33s}
$e(\SL_n(3)) \leq 6$.
\end{lemma}

\begin{proof} Let $h\in H=\SL_n(3)=\SL(V)$, and $V=\bigoplus_{i=1}^t V_i$, where $V_i$ are indecomposable ${\mathbb F}_3\langle h\rangle$-modules, and let $h_i$ be the projection of $h$ in $\GL(V_i)$ for all $i$. By Lemma \ref{id4}, the product of some three $\SL(V_i)$-conjugates of $h_i$ is of order $1$ or $ 2$ in $\SL(V_i)$ for $i=1,\ldots,k$;  hence a product of $3$ conjugates of $h$ in $H$ is of order 2. Then $e(h)\leq 6$.
\end{proof}

\begin{lemma}\label{s53} Let $G=\PSL_5(3)$. Then $e(G)\leq 5$. In addition, $e(\PSL_3(3))=e(\PSL_4(3))=3$.\end{lemma}

\begin{proof}
The second claim follows from \cite{gap}, so we  prove that $e(G) \leq 5$. Let $g=su \in G$, where $s$ is semisimple, $u$ is unipotent and $su=us$. If $u=1$ is  then $e(g)\leq 5$
by Lemma \ref{un2}. If $s=1$  then $e(g)\leq 2$ by \cite[Theorem 1.8(ii)]{TZ04}. 

Suppose that  $u,s\neq 1$.   Let $h\in H=\SL_5(3)$ be such that $g=hZ(H)$, and $h=su$, where $s,u\in  \langle h\rangle$, $s$ is  semisimple, $u$ is unipotent. Then $s$ is not scalar. Recall that an $(n \times n)$-matrix $A$ over a field $F$ is called cyclic if there exists $x \in F^n$ such that $\{x,Ax,\ldots,A^{n-1}x\}$ is a basis of $F^n$. The matrix $A$ is cyclic if and only if its Jordan canonical form contains only one Jordan block for every eigenvalue.

We   use Lemma  \ref{centsem} and the notation therein in the following discussion, in which we list the possibilities for $u$ (up to conjugation). In particular, $n_1,\ldots ,n_k$ are the dimensions of the homogeneous components of $s$; each of them is $g$-stable, and $m_i$ are the common dimensions of irreducible constituents of $s$ on $V_i$.
We denote by $J_m$ the canonical Jordan block of size $m$. 
Since $s$ is not scalar, $k \geq 2$ and $k<5$ since $u\neq 1$.
 
(a) $u=\diag(J_4,1)$. The fact that $u \in C_G(s)$ implies that $k=2$ 
so $m_1=m_2=1$, $n_1=4$, $n_2=1$ or conversely. Therefore $h$ is a cyclic matrix whose all eigenvalues are in $\mathbb{F}_3$. By \cite[Theorem 3]{Lev94}, $g^{-1}$ is the product of two conjugates of $h$, whence $e(g)\leq 3$.

(b) $u=\diag(J_3,J_2)$. Then $m_1=m_2=1$ and we can assume that $n_1=3$,   $n_2=2$,  $s=\diag(\mu_1 \cdot \Id_3,\mu_2 \cdot \Id_2)$, 
$\mu_1,\mu_2\in \mathbb{F}_3$ and $\mu_1^3\mu_2^2=1$. As $\mu_1\neq \mu_2$, $h$ is a cyclic matrix. As above,  $e(g)\leq 3$ by \cite[Theorem 3]{Lev94}.

(c) $u=\diag(J_3,1,1)$. In this case 
$s=\diag(\pm \Id_3,s_2)$, $s_2\in \GL_2(3)$. 
Then $J_3$ is real in $\SL_3(3)$ and $s_2$ is real in $\PGL_2(3)$. 
So the product of some two conjugates of $g$ in $G$ is $\diag( \pm \Id_3,\pm \Id_2)$, whence  
 $e(g)\leq 4$.

(d) $u=\diag(J_2,1,1,1)$. In this case we can assume that 
$s=\diag(\mu \cdot \Id_2,s_2)$, $\mu=\pm 1$, $s_2\in \SL_3(3)$.
 If $s_2$ is regular, then the product of  3 conjugates of $s_2$ in $\SL_3(3)$ equals $ \Id_3$ by Lemma \ref{Gow}, as $\det (-\Id_3)=-1$. If $\mu=-1$ then, by Lemma \ref{s33}, the product of 3 conjugates of $J_2$ in $\SL_2(3)$ equals $-\Id$, so the product of some $3$ conjugates of $-J_2$ in  $\SL_2(3)$ equals 1. Then  $e(g)\leq 3$.  If $\mu=1$ then, since
$J_2^3=1$, again $e(g)\leq 3$. If $s_2$ is not regular in $\SL_3(3)$ then $s_2$ centralizes some unipotent element in $\SL_3(3)$, and hence it is diagonalizable with at least two equal eigenvalues which belong to $\mathbb{F}_3$. So we can assume that $s_2=\diag(1,\pm \Id_2)$. Then $g = \diag(\mu.J_1,1,\pm \Id_2)$ is contained in the group $X= \left\{ \diag(A,\det A^{-1},\pm\Id_2)\ |\ A \in \GL_2(3) \right\}$. It follows that $g$ is real in $X$, so $e(g)=2$.
 
(e) $u=\diag(J_2,J_2,1)$. Then $n_1/m_1 \geq 2$. If $m_1 \geq 2$, then $m_1=2$ and $n_1=4$, while if $m_1=1$ then $n_1 \geq 2$. This implies that 
there are two isomorphic irreducible components, and their dimension is $1$ or $2$.
Suppose first that these are of dimension $2$. Then, replacing $g$ by some conjugate, we can assume that $s=\diag(s_1,s_1,t)$, where $s_1\in \GL_2(3)$ is irreducible, $t\in \mathbb{F}_3$. Then $t=1$ as $ \det s_1^2=1$. So $s=\diag(s_1,s_1,1)$.  
We deduce that $h=\diag(A,1)$, where $A=s_2\otimes J_2 \in \GL_2(3)\otimes \GL_2(3)$. Note that $\GL_2(3)\otimes \GL_2(3)\subset \SL_4(3)$.  
Since $J_2$ is real in $\GL_2(3)$, and semisimple elements are real in  $\GL_2(3)$ too, the result follows. Suppose now that $s$ has at least $2$ isomorphic composition factors   of dimension $1$. Then some eigenvalue $\mu$ of $h$ is of multiplicity at least 2. Let $W$ be the $\mu$-eigenspace of $h$. If $\dim W=2$ or $3$ then $h=\diag(-J_2,J_2,1)$ 
and  $e(h)=2$ as $J_2$ is real in $\GL_2(3)$. If $\dim W=4 $ then  $h=\diag(-J_2,-J_2,1)$, and again  $e(h)=2$.\end{proof} 

\begin{lemma}\label{44s} Let $G=\SL_n(4)$. Then $e(G)\leq 18$. If $h\in G$
is semisimple then $e(h)\leq 9$.
\end{lemma} 

\begin{proof} 
Suppose that $h \in G$ is semisimple. Let $V=\bigoplus_{i=1}^k V_i$ be a decomposition of $V$ as a direct sum of minimal $\langle h \rangle$-stable subspaces of $V$. Let $h_i$ be the projection of $h$ on $V_i$ for $i=1,\ldots ,k$.  Then $h_i$ is irreducible on $V_i$. Note that $|\GL(V_i) : \SL(V_i) Z(\GL(V_i))| \in\{1,3\}$. By Corollary  \ref{GuT},
there are $x_i,y_i\in \SL(V_i)$ such that $h_ix_ih_ix_i^{-1}y_ih_iy_i^{-1}\in Z(\GL(V_i))$.
This is also true for  $i$ with $\dim V_i= 1$.
It follows that the product of some 3 conjugates of $h$ in $H$ acts scalarly on each $V_i$. 
Since $q=4$, we have  $e(h)\leq 9$.

In general, let $h=su$, where $1\neq s\in H$ is semisimple, $1\neq u\in H$ is unipotent and $us=su$. Then $e_{C_G(s)}(u)\leq 2$. Indeed, $C_{\GL(V)}(s)$ is a direct product of groups isomorphic to $\GL_{m}(q^k)$ for various $m,k$, see \cite[Lemma 3.2]{TZ04}. Therefore, $u$ is contained in a subgroup of $C_{\GL(V)}(s)\cap H$ isomorphic to a direct product of groups $\SL_{m}(q^k)$, since $|u|$ is a $p$-element for $p|q$ and $|GL_{m}(q^k)/\SL_{m}(q^k)|$ is a $p'$-group. By \cite[Theorem 1.7]{TZ04}, $u$ is real in $C_G(s)$. Therefore, $s^2$ is a product of 2 conjugates of $h$; as $e(s^2)\leq 9$, we have $e(h)\leq 18$.
\end{proof}

\begin{prop}\label{ppn} Let $G=\PSL_n(q)$, $n>2$. Then $e(G)\leq n$.
\end{prop}
\begin{proof} For $q>3$ this follows from Lev's result on $\cn(G)$, see Lemma \ref{Lev}. If $q=2$ the result follows from Theorem \ref{gg6} and Lemma \ref{s25}. If $q=3$ the result follows from Lemmas \ref{33s} and \ref{s53}.   
\end{proof}

\begin{prop} \label{geqn}
Let $F$ be an infinite field and let $n \geq 2$ be an integer. If $n \in \{2,3\}$, assume that $F$ is algebraically closed. Then $e(\PSL_n(F)) = n$.
\end{prop}

\begin{proof}
Assume first that $n=2$, so that $F$ is algebraically closed. As $\GL_2(F) = \SL_2(F) \cdot Z(\GL_2(F))$, elements conjugate   in $\GL_2(F)$ are conjugate in $\SL_2(F)$. For $\alpha \in F^{\times}$, set
$$M_{\alpha} = \left( \begin{array}{cc} \alpha & 0 \\ 0 & \alpha^{-1} \end{array} \right), \hspace{.2cm} 
A = \left( \begin{array}{cc} 1 & 1 \\ 0 & 1 \end{array} \right), \hspace{.2cm}
S = \left( \begin{array}{cc} 0 & 1 \\ 1 & 0 \end{array} \right), \hspace{.2cm}
T = \left( \begin{array}{cc} 1 & 0 \\ 0 & -1 \end{array} \right).
$$
By basic linear algebra, since $F$ is algebraically closed, every matrix in $\SL_2(F)$ is conjugate to one of the matrices $A$, $-A$ or $M_{\alpha}$ for some $\alpha \in F^{\times}$. We have $S^{-1}M_{\alpha}S = M_{\alpha^{-1}} = M_{\alpha}^{-1}$, $T^{-1}AT=A^{-1}$. So $e(\PSL_2(F))=e(\SL_2(F))=2$.

Assume now that $n \geq 3$ and let $l:=e(\PSL_n(F))$. 
We have $l \leq n$ by Lev's result (Lemma \ref{Lev}), so we only need to prove that $l \geq n$. Assuming this is false, Lemma \ref{lemmaPSL} implies that $\mu^{n(n-1)}=1$ for every $\mu \in F^{\ast}$. This contradicts the fact that $F$ is infinite.
\end{proof}

We summarize the information obtained in this section as follows.

\begin{corol} \label{boundsPSL}
Let $n \geq 2$ be an integer, let $F$ be a field and let $G:=\PSL_n(F)$.
\begin{enumerate}[$(1)$]
\item If $F$ is finite of order $q$ and $d=(n,q-1)$, then $e(G) \geq \min \{n,(q-1)/d\}$.
\item If $F$ is finite of order $q$ and $3 \leq n \leq (q-1)/d$ then $e(G)\leq n$.
\item If $n \geq 4$ and $F$ is infinite, then $e(G)=n$.
\item If $n\leq 3$ and $F$ is algebraically closed, then $e(G)=n$.
\item $e(\PSU_n(q)) \geq \min\{n,(q+1)/(n,q+1)\}$.
\end{enumerate}
\end{corol}


\begin{proof} Item (1) follows from Lemma \ref{lemmaPSL}. Item (2) follows from 
Lev's theorem (Lemma \ref{Lev}). 
Item (5) follows from Lemma \ref{lemmaPSL} by choosing $\mu$ of order $q+1$, so 
the element $g$ in Lemma \ref{lemmaPSL} belongs to $\SU_n(q)$. Items (3) and (4) follow from Proposition \ref{geqn}.
\end{proof}

\subsection{Unitary groups   ${\rm PSU}_n(q)$}\hfill\\

We first recall a few well known facts on maximal parabolic subgroups of $H=\SU_n(q)=\SU(V)$.
By \cite[Proposition  2.3.2]{KL}, $V=W+X+W'$, where $W,W'$ are totally isotropic subspaces of $V$ of dimension $r=[n/2]$ and $(W+W')^\perp=X$. This is equivalent to saying that  $V$
has a basis with Gram matrix $\begin{pmatrix}0&0&\Id_r\\ 0&1&0\\ \Id_r&0&0\end{pmatrix}$ if $n$ is odd and 
 $\begin{pmatrix}0&\Id_r\\  \Id_r&0\end{pmatrix}$ if $n$ is even. Let $P$ be the stabilizer of $W$ in $H$. Then $P=LU$, where $U$ is the unipotent radical of $P$, and $L$
is a complement of $U$ in $P$.  Let $\sigma$ be the automorphism of $H$ arising from the Galois automorphism of $\mathbb{F}_{q^2}/\mathbb{F}_q$, sending $x \in \mathbb{F}_{q^2}$ to $x^q$, so $\sigma$ rises every matrix entry of $h\in H$ to the $q$-th power. Under the basis with above Gram matrix, $L$ consists of  all matrices of determinant $1$ of the form $\mbox{diag}(M,d,{}^t\sigma(M)^{-1})$ if $n$ is odd, and $\mbox{diag}(M,{}^t\sigma(M)^{-1})$ if $n$ is even, where $M\in \GL_r(q^2)$ is such that $\det (M) \cdot \sigma(\det (M^{-1})) \cdot d=1$. (Here $t$ means the transpose.) 
In particular,  if  $M$ is diagonal then $\sigma(M)=M^q$. It follows that the derived subgroup of $L$ is $L' \cong \SL_r(q^2)$.

\begin{lemma} \label{pu3} Let $H=\SU_n(2)$, $n>1$ and let $s\in H$ be a semisimple element.   Then  $e(s)\leq 9$ and $e(H)\leq 18 $.\end{lemma}

\begin{proof} Let $V$ be the underlying space for $H$. By \cite[Satz 2]{Huppert}, $V=\bigoplus_i V_i$, where $V_i$'s are minimal $\langle s \rangle$-stable non-degenerate subspaces of $V$. Let $s_i$ be the restriction of $s$ to $V_i$. Then either $s_i$ is irreducible, or $V_i=W_i \oplus W_i'$,
where $W_i,W_i'$ are totally isotropic minimal $\langle s_i \rangle$-stable subspaces. In addition, there exists a basis of $V_i$ such that if $M_i$ is the matrix of $s_i$ on $W_i$ then 
the matrix $M_i'$ of $s_i$ on $W_i'$ is ${}^t\gamma(M_i)^{-1}$, where $\gamma$ is a non-trivial Galois automorphism of ${\mathbb F}_4$ (\cite[Satz 2]{Huppert}). 
As $W_i$ is irreducible, $s_i$ is a regular element of $\GL(W_i)$. Moreover, by Schur's lemma $s_i$ is regular in $\GL(V_i)$ unless $W_i,W_i'$ are isomorphic $\mathbb{F}_{q^2}\langle s_i\rangle $-modules, equivalently, $M_i$ and $M'_i$ are similar as matrices of $\GL_d(q^2)$, $d=\dim W_i$. Then $M_i$ is contained in a subgroup $X_i$ of  $\GL_d(q^2)$ isomorphic to $\U_d(q)$. It follows that $V_i$ is a direct sum of  two $\langle s_i \rangle$-stable non-degenerate subspaces. Therefore, if we choose the decomposition $V=\bigoplus_i V_i$ with maximal number of non-degenerate terms $V_i$ then $s_i$ is a regular element in $\GL(V_i)$ and $\U(V_i)$.

Since $|\U(V_i)/Z(\U(V_i)) \SU(V_i)|\in\{1,3\}$, by Corollary \ref{GuT} there are $x_i,y_i\in \SU(V_i)$ such that $s_i x_i s_i x_i^{-1} y_i s_i y_i^{-1} \in Z(\U(V_i))$. Therefore, the product of some 3 conjugates of $s$ is of order 3 or 1.  This implies $e(s)\leq 9$. 

Let $h\in H$ and $h=su$, $s,u\in \langle h\rangle$ with $s$ semisimple, $u$ unipotent. 
Note that $C_{\U(V)}(s)$ is a direct product of groups
isomorphic to $\GL_{m}(q^{2k})$ and $\U_m(q^k)$ for various $m,k$, see \cite[Lemma 3.3]{TZ04}. Therefore, $u$ is contained in a subgroup of $C_{\GL(V)}(s)\cap H$ isomorphic to a direct product of groups  $\SL_{m}(q^k)$ and  $\SU_m(q^k)$. By \cite[Theorem 1.7(ii)]{TZ04}, applied to every multiple, $u$ is real in $C_G(s)$. Therefore, $s^2$ is a product of 2 conjugates of $h$; as $e(s^2)\leq 9$, we have $e(h)\leq 18$. \end{proof}

\begin{lemma} \label{pu2} Let $G=\PSU_n(q),n>1$, q odd.  
Then $e(G)\leq 3n+3$. If $(n,q)=1$ then $e(G)\leq 3n$.\end{lemma}

\begin{proof} Let $g\in G$ and let $h\in H = \SU_n(q)$ be such that $g=hZ$ where $Z=Z(H)$. Let $h=su$ with $s,u\in \langle h\rangle$, $s$ semisimple, $u$ unipotent.  
By Lemma \ref{u26}, $e_{C_H(s)}(u)\leq 3$. Therefore, $s^3$ is a product of $3$ conjugates of $h$ in $C_H(s)$. By Lemmas \ref{un2} and \ref{u42}, $e_G(s^3Z) \leq n$ if $(n,q)=1$ and  $e_G(s^3Z) \leq n+1$ otherwise. So the result follows by Lemma \ref{gr1}.\end {proof}

\begin{theor}\label{u36} 
Let $G = \PSU_n(q)$, $n>2$. 
If $q$ is even, then  $e(G) \leq 2n-2$.
If $3|q$ then $e(G) \leq 3n-3$.
\end{theor} 

\begin{proof} If $2|q$ then the exponent of a Sylow $2$-subgroup of $G$ equals $2^m$, where $n\leq 2^m\leq2n-2$. Similarly, if $3|q$ then the exponent of a Sylow $3$-subgroup of $G$ equals $3^m$, where $n\leq 3^m\leq 3n-3$. The result follows from Lemma \ref{5eg}.\end{proof}

\section{Exceptional groups}

By Lemma \ref{ee7}, we are left to review the following exceptional groups: ${\rm Sz}(q), p=2$; $F_4(q), p=2$; ${^2}F_4(q), p=2$; $E_6(q)$, ${^2}E_6(q)$; $E_7(q), p=2$; $E_8(q), p \in \{2,5\}$. For $G=G_2(q), p=2$, we obtain a better bound than in Lemma \ref{ee7}. 



\begin{lemma} \label{g24}
Let $G=G_2(q)$, q even. Then $e(G)=3$. 
\end{lemma}

\begin{proof} The character table of $G$ is available in \cite{EY}. Table  \ref{tableG2}
below is a part of the character table given there. The irredicible characters of $G$ are partitioned as sets labeled in the 1-st column, the last column contains the number of the characters in the corresponding set. We only need the character values at the non-real conjugacy classes. These are exactly the classes whose charater values are not all   real numbers.

By inspection of \cite{EY} one observes that the irreducible character values are not real only at two conjugacy classes, denoted by  $B_2(1)$ and $ B_2(2)$. So $e(g)=2$ unless $g\in B_2(1)\cup B_2(2)$.


We do the calculation for $g\in B_2(1)$, since the one for $B_2(2)$ is similar. The characters that take nonzero value on $g\in B_2(1)$ are  in Table  \ref{tableG2} where $q=2^n$, $ \varepsilon=(-1)^n$, $\omega$ is a primitive $3$-root of unity so its absolute value equals 1 and $\omega_k = \omega^k+\omega^{-k}$. So the absolute value of $\omega_k$ is at most $2$.  We split the following sum grouping the first 6 nontrivial character lines and the last 8 lines above (type $\theta$ and type $\chi$). We have

$$\left| \sum_{1 \neq \chi \in \mbox{Irr}(G)} \frac{\chi(g)^3}{\chi(1)} \right| \leq \sum_{1 \neq \chi \in \mbox{Irr}(G)} \frac{|\chi(g)|^3}{\chi(1)} \leq \frac{7(4q/3)^3}{q^5/6} + \frac{8(3q^2/4+5q/2)}{q^5/2} < 1$$
for $q \geq 16$. The calculation for small $q$ can be done by hand.
\end{proof}

\begin{theor} \label{sf4} 
For $q$ even let  $G\in\{\Sz(q),F_4(q),{}^2F_4(q)\}$ be a simple group. Then $e(G)\leq 3,4,4$, respectively.\end{theor}

\begin{table}[ht]
  \centering
	  \caption{Nonzero irreducible $G_2(q)$ character values for class $B_2(1)$}\label{tableG2}
  {\fontsize{9}{11}\selectfont
	\begin{tabular}{|c|c|c|c|}
      \hline
			$\chi$ & $\chi(g)$ & $\chi(1)$ & Number of characters \\ \hline \hline
$\theta_0$ & 1 & 1 & 1 \\ \hline
$\theta_1$ & $(-1/3)(q-1)$ & $(1/6)q(q+1)^2(q^2+q+1)$ & 1 \\ \hline
$\theta_1'$ & $(-1/3)(q-1)$ & $(1/6)q(q-1)^2(q^2-q+1)$ & 1 \\ \hline
$\theta_3$ & $(1/3)(q-1)$ & $(1/3)q(q^4+q^2+1)$ & 1 \\ \hline
$\theta_4$ & $(1/3)(q+2)$ & $(1/3)q(q^4+q^2+1)$ & 1 \\ \hline
$\theta_8$ & 1 & $q(q+\varepsilon)(q^3+\varepsilon)$ & 1 \\ \hline
$\theta_9(k)$ & $(1/3)(q-1)+q \omega^k$ & $(1/3)q(q^2-1)^2$ & 2 \\ \hline \hline
$\chi_1(k)$ & 1 & $(q^3+1)(q^2+q+1)$ & $(q-3-\varepsilon)/2$ \\ \hline
$\chi_1'(k)$ & $-1$ & $(q^3-1)(q^2-q+1)$ & $(q-1+\varepsilon)/2$ \\ \hline
$\chi_2(k)$ & 1 & $q(q^3+1)(q^2+q+1)$ & $(q-3-\varepsilon)/2$ \\ \hline
$\chi_2'(k)$ & -1 & $q(q^3-1)(q^2-q+1)$ & $(q-1+\varepsilon)/2$ \\ \hline
$\chi_3(k)$ & $\omega_k$ & $(q^3+1)(q^2+q+1)$ & $(q-2)/2$ \\ \hline
$\chi_5(k,l)$ & $\omega_{k+l}$ & $(q+1)(q^3+1)(q^2+q+1)$ & $(q-3-\varepsilon)(q-5+\varepsilon)/12$ \\ \hline
$\chi_6'(k)$ & $-\omega_k$ & $q^6-1$ & $q(q-2)/4$ \\ \hline
$\chi_7(k)$ & $\omega_k$ & $(q-1)(q^2-1)(q^3+1)$ & $(q-\varepsilon)(q+1+\varepsilon)/6$ \\
			\hline
\end{tabular}
  }
\end{table}

\begin{proof}For the Suzuki groups $G={\rm Sz} (2^{2m+1})$, $m>0$, we have $e(G)=\cn(G)=3$ by \cite[Theorem 4.1, p.239]{AH} as  $G$ contains nonreal elements.

Suppose now that $G \in \{F_4(q),{}^2F_4(q)\}$. By \cite[Proposition  10.11]{TZ04}, $G$ has property $\frac{1}{2}\mathcal{SR}^*$. This means (see \cite[p. 333]{TZ04}) that the unipotent elements $u \in C_G(s)$ are semirational. Indeed, the property $\mathcal{SR}^*$ means that  $u,u^m$ are conjugate in $C_G(s)$ whenever $(m,|u|)=1$ and $4|(m-1)$. This is equivalent to saying that $u^m$ is conjugate to $u$ or to $u^{-1}$ in $C_G(s)$. In particular, every unipotent element $u$ of $G$ is semirational in $G$. So, by Corollary \ref{cor881}, $e(u)\leq 4$. 

Suppose that $g\in G$ is not unipotent, and write $g=su$, where $s$ is semisimple, $u$ is unipotent and $su=us$. By Lemma \ref{sn2}, $s$ is real. So  $e(g)=e(s)=2$ if $u=1$.

Suppose that $s\neq 1,u\neq 1$. Suppose that $s\neq 1,u\neq 1$. As observed in \cite[pp. 381-382]{TZ04}, either $u$ is real in $C_G(s)$ and then $e(g) \leq 4$ by Lemma \ref{xy8} or $G={}^2F_4(q)$ and $C_G(s)\cong H\times C$, where $H\cong \Sz(q)$ and $C$ is a cyclic group. 
Thus, if $G=F_4(q)$ then $e(G)\leq 4$.

Suppose that $G={}^2F_4(q)$. 
We  prove that $g^2$ is real, which implies the result in this case.

  Shinoda \cite{Shi} describes the conjugacy classes and element centralizers in $G$. The non-trivial non-regular semisimple elements of $G$, in particular $s$, are of type $t_i$ with $i=1,2,4,5,7,9,$ in notation of \cite{Shi}. 

Assume first that $3$ divides $|s|$. Then $g\in M$, where $M$ is a subgroup of $G$ described in  \cite[Proposition 1.2, items (1),(2)]{Ma}. In item (2) $M=(Z_{q^2+1} \times Z_{q^2+1}):\GL_2(3)$, the
elements of order $2^m3^k$ with $m,k>1$ are in $\GL_2(3)$, and hence  real (as $q^2+1$ is coprime to $6$). 
In item (1) $M=\SU_3(q).2$ so $g^2\in \SU_3(q)$. 

Suppose first that $s \in Z(\SU_3(q))$. Then $u\in \SU_3(q)$ so $|u^2|\leq 2$.
If $u^2=1$ then $g^2$ is semisimple, hence real. Suppose that $u^2\neq 1$. 
One easily observes that the involutions are conjugate in $\SU_3(q)$. 
In addition,  $s$ is real in $M$ (see the proof of Proposition 1.2 in \cite{Ma}). 
Then $xsx^{-1} = s^{-1}$ for some $x\in M$. Then $xg^2x^{-1} = s^{-2} \cdot xu^2x^{-1}$ and $xu^2x^{-1}$ is an involution. So $yxu^2x^{-1}y^{-1} = u^{-2}$ for some $y \in \SU_3(q)$. Then $yxg^2x^{-1}y^{-1}=s^{-2}u^{-2}$, that is $g^2$ is real.

Suppose that $s\notin Z(\SU_3(q))$. Then $s^2$ is not scalar. If $u^2=1$ then $g^2$ is semisimple and hence real.
Assume that $u^2\neq 1$. As $u^2,s^2\in \SU_3(q)$, we have  $g^2\in \SU_3(q)$.
The element $g^2$ is conjugate in $\GL_3(q^2)$ to a matrix $\diag(a,a,a^{-2})$ for some $a\in \mathbb{F}_{q^2}$, $a\neq a^{-2}$. (The latter implies $q>2$.) Let $W$ be the underlying space for $\SU_3(q)$, and let $W_0$ be the $a^{-2}$-eigenspace of $g^2$ on $W$. Then $\dim W_0=1$.

We claim that $g^2$ acts on $W/W_0^{\perp}$ as the multiplication by $a^{2q}$. Indeed, if $w \in W_0$, $v \in W$ then $w=a^2a^{-2}w=a^2 g^2w$ so since $a^{q^2}=a$, and $\sigma(a)=a^q$, where $\sigma$ is the Galois automorphism of $\mathbb{F}_{q^2}/\mathbb{F}_{q}$ of order 2. We have
\begin{align*}
(w,g^2v-a^{2q}v) 
& = (w,g^2v)-(w,a^{2q}v) = (a^2g^2w,g^2v) - \sigma(a^{2q}) (w,v) \\
& = a^2(g^2 w,g^2v) - a^{2q^2} (w,v) = a^2(w,v) - a^2(w,v) = 0.
\end{align*}
and hence $g^2v-a^{2q}v \in W_0^{\perp}$, proving the claim. Then $a^{2q}=a$ and $a^{2q-1}=1$. Since $a^{q^2-1}=1$, we deduce that $1 < (2q-1,q^2-1) = (2q-1,q+1) \leq 3$, therefore $(2q-1,q^2-1)=3$ and hence $a^3=1$, a contradiction.

Next suppose that $(|s|,3)=1$. Then $\langle s\rangle$ contains an element $h$, say, of a prime order $r>3$.  Then $g\in M$, where $M$ is one of the subgroups described in \cite[Proposition 1.3, items $(1)-(5)$]{Ma} (as $u\neq 1$). In these items 
$M= N_G(\langle h\rangle )$, and $h$ is real in $M$. In items (2),(3) we have $u^2=1$,
so $g^2$ is semisimple, and hence real.  In items (1),(4),(5) we have 
$M = H\times T$, where $H \cong \Sz(q)$ and $T$ has a cyclic normal subgroup of odd order and index $2$ or $4$.
If $u$ is an involution then $g^2$ is semisimple and $e(g)\leq 4$.
Otherwise, $|u|=4$ and $u^2$ is an involution. Then $g^2 = s^2u^2$,
where  $u^2 = u_1u_2$,  $u_1\in T^2$, $u_2\in H$, 
and $u_1s = su_1$ so the group $\langle s,u_1\rangle$ is cyclic. (In item (1) $u_1=1$.) 
Note that the involutions of $\Sz(q)$ are conjugate \cite[Proposition 7]{Su} as well as those of $T^2$ if $|T^2|$ is even.

Suppose first that $s\in T$.   As $s$ is real, we have  $xsx ^{-1}=s^{-1}$ for some $x\in T$. 
If $u_1=1$ then $u^2$ is an involution of $H$ hence $g^2$ is real. If $u_1\neq 1$ then 
 $T$ has an element of order $4$, so $|T^2|$ is even and
$T^2=\langle u_1,T^4\rangle$ so $s\in Z(T^2)$.
If $|T^2|$ is even then  the involutions of $T^2$ are conjugate in $T^2$ and
the conjugation elements commute with $s$. Therefore, the involutions of 
$C_M(s)$ are conjugate in  $C_M(s)$. This again implies that $g^2$ is real.

Suppose that $s\not\in T$. Then $s=s_1s_2$ for some $s_1\in T,1\neq s_2\in H$. Then $C_H(s_2)$ consists of odd order (hence semisimple) elements, see \cite[Prop 1,2]{Su}. So $u\in T$. As above, we conclude that $g^2$ is real.
\end{proof}

The group $G={}^2F_4(2)$ is not simple, but the derived group  ${}^2F_4(2)'$ is simple. For this group we have:

\begin{lemma}\label{2sf4} Let $G={}^2F_4(2)'$. Then $e(G)= 3$.\end{lemma}

\begin{proof} 
By \cite[p. 75]{at}, if $g\in G$ and $|g|\neq 8,16$ then $g$ is real so $e(g)=2$. 
More precisely, if $g$ is not real then $g$ is in one of the classes $8A$, $8B$, $16A$, $16B$, $16C$, $16D$. Using \cite[p. 75]{at}, it is easy to check that $\sum_{\chi \in \mbox{Irr}(G)} \chi(g)^3/\chi(1) \neq 0$, so $e(g)=3$.
\end{proof}
 
We now discuss the groups of type $E_6$, $E_7$, $E_8$. 

For every quasisimple group of Lie type $G$ in defining characteristic $p > 0$ let $U$ be a Sylow $p$-subgroup of $G$ and $B = N_G(U)$. Then $B = HU$, where $H \cap U = \{1\}$. By \cite[Proposition 2.5.5]{C}, there  exists a conjugate $U^-$ of $U$ such that $N_G(U) \cap N_G(U^-) = H$. By \cite[Theorem H, p. 344]{egh}, 
for every non-central conjugacy class $C$ of $G$ and every $h\in H$ there exists 
$u_1\in U^-$, $u_2\in U$ such that $u_1 h u_2 \in C$. In particular, $C \cap  U^-U$ is not empty. 

\begin{lemma} \label{e62}
Let $G\in \{E_6(q), {^2}E_6(q)\}$ be a simple group, and $p|q$ a prime. Then $e(G)\leq d$, where $d=p$ for $p>11$, and $16,27,25,49,121$ for  $p=2,3,5,7,11$, respectively. \end{lemma}

\begin{proof} Note that the exponent of a Sylow $p$-subgroup of $G$ equals $p$
for $p>11$, $p^2$ for $3<p\leq 11$, and $16,27$ for $p=2,3$  \cite{Lw}. So the result follows from  Lemma \ref{5eg}.\end{proof}

We are indebted to Frank L\"ubeck (University of Aachen)
for the following fact (obtained by computations with the computer package ``CHEVIE''):

\begin{lemma} \label{613} 
Let $G \in \{E_6(q), {^2}E_6(q)\}$. Then $G$ contains a regular semisimple element of order $13$; if $(q,6)=1$ then G contains a regular semisimple element of order $12$.
\end{lemma}

\begin{theor} \label{666} 
Let $G\in \{E_6(q), {^2}E_6(q)\}$, $q=p^k$ for a prime $p$. Then $e(G) \leq 16,27,25$ for $p=2,3,5$, respectively,  $e(G)\leq p$ for $13 \leq p \leq 31$ and $e(G)\leq 36$ for $p=7,11$ and $p>31$. If $s\in G$ is a semisimple element, then $e(s)\leq 12$ if $(q,6)=1$ and $e(s) \leq 13$ otherwise. 
\end{theor}

\begin{proof} For $p\leq 31$ see Lemma \ref{e62}.  By Lemmas \ref{Gow} and \ref{613}, $s$ is a product of two elements of order $13$, and even $12$ if $(q,6)=1$. Then $e(s)\leq 13,12$, respectively.  

In general, let $g=su\in G$ be the Jordan decomposition of $g$, and let $C$ be the conjugacy class of $G$ containing $g$. By Lemma \ref{u5n},
the generators of $\langle u\rangle$ lie in at most two conjugacy classes of $C_G(s)$. Therefore, $u$ is either real or semirational in $C_G(s)$, hence $e_{C_G(s)}(u)\leq 3$. By Corollary \ref{cor881},  $s^m\in C^m$ for $m=2$ or $3$. Now we apply Lemma \ref{xy8}. If $(q,6)=1$ then $e(s) \leq 12$, so $e(g) \leq 36$. If $(q,6) \neq 1$ then the claim follows from Lemma \ref{e62}.\end{proof}

\begin{lemma}\label{e88}  Let $\mathbf{G}$ be the simple algebraic group of type $ E_8$ in characteristic $p>0$ and let $1\neq s\in  \mathbf{G}$  be a semisimple element.

$(1)$ The group $\mathbf{C}:=C_{\mathbf{G}}(s)$ is connected reductive, $\mathbf{C}=Z(\mathbf{C})\cdot \mathbf{G}_1\cdots \mathbf{G}_m$ for some simple algebraic groups $\mathbf{G}_1\,\ldots ,\mathbf{G}_m$;

$(2)$ the groups $\mathbf{G}_i$, $1\leq i\leq m$, are either of exceptional type distinct from $E_8$, or of classical groups of rank at most $8$.\end{lemma}
 
\begin{proof} (1) $\mathbf{C}$ is reductive (that is, it contains no non-trivial unipotent normal subgroup), see \cite[3.9 on p. 197]{Sbo}, and connected  by \cite[Theorem 14.16(i) or 14.20]{MTe}, since the simple algebraic group of type $ E_8$ is simply connected. So the claim is a well known fact of general theory of algebraic groups, see for instance \cite[\S 1.8, pp.16-17]{C}.
 
(2) A maximal torus of $\mathbf{C}$ is the product of some maximal tori of the multiples $Z(\mathbf{C})$, $\mathbf{G}_1$, $\ldots$, $\mathbf{G}_m$, and the center of a simple algebraic group is a finite group. Therefore, the rank of $\mathbf{C}$ equals the sum of the ranks
of the multiples, and hence the rank of every group $\mathbf{G}_i$ is at most $8$. 
If some $\mathbf{G}_i$ is of type $E_8$
then $m=1$ and the rank of 
$ Z(\mathbf{C})$ equals $0$; as $\mathbf{C}$ is connected,  $ Z(\mathbf{C})$ is contained in 
$\mathbf{G}_1$, so $\mathbf{C}$ is a simple algebraic group of type $E_8$. 
In this case $Z(\mathbf{C})=1$, hence $s=1$, a contradiction. 
\end{proof}

\begin{theor} \label{ue67} If $G$ is of type $E_7(q)$ or  $E_8(q)$ and $q$ is even then $e(G)\leq 4$, $e(G)\leq 8$ respectively. If $G=E_8(q)$ and $5$ divides $q$, then $e(G) \leq 6$. \end{theor}

\begin{proof} Let $p|q$ be a  prime. There is a maximal parabolic subgroup $P$ of 
$G$ such that, for $U=O_p(P)$, we have $(P/U)' \cong E_6(q)$ or $\Omega^+_{14}(\mathbb{F}_q)$, respectively, for $E_7(q)$, $E_8(q)$. Indeed, by general theory of parabolic subgroups of simple algebraic groups  \cite[\S 12.1]{MTe}, the parabolic subgroups of $G$ are determined by subsets of the nodes at the Dynkin diagram $D$ of $G$, and the semisimple component of $P/U$ corresponds to the subdiagram obtained by removing from $D$ the nodes determining $P$. In case of $E_7,E_8$ use the node 7, 1, respectively. Then the subdiagram obtained is of type $E_6, D_7$, respectively. 

Suppose that  $q$ is even. Then the universal group of 
type $D_7(q)$ has trivial center, so this is isomorphic to $\Omega^+_{14}(\mathbb{F}_q)$.  By \cite[Theorem 1.8(ii,iii)]{TZ04}, every unipotent element of $(P/U)'$ is real. 
 
Note that the index of $(P/U)'$ in $P/U$ is odd. For this observe that $P=LU$, a semidirect product, where $L\cong P/U$ is a Levi subgroup of $P$ (see \cite[Proposition 12.6.3]{MTe}), and $P=LU$.  
Let $\mathbf{G}$ be the simple algebraic group such that $G=\mathbf{G}^F$ for a Frobenius endomorphism $F$ of $\mathbf{G}$. By \cite[Proposition  26.1]{MTe},
$P=\mathbf{P}^F$ and $L= \mathbf{L}^F$ for a suitable parabolic subgroup $\mathbf{P}$
of $\mathbf{G}$ and a Levi subgroup $\mathbf{L}$ of $\mathbf{P}$.  
By \cite[Proposition  12.6]{MTe}, $\mathbf{L}$ is a connected reductive group.  Note that a simple algebraic group of type $E_8$
 is   of simply connected type \cite[Table 9.2]{MTe}. As  $q$ even,  a simple algebraic group of type  $E_7$ is isomorphic to that of simply connected type.  So we can assume this of simply connected type too.   Then   $\mathbf{L}'$  is   of simply connected type \cite[Proposition  12.14]{MTe}.  
Therefore, $(L')^F$ is generated by the unipotent elements of $L^F$  \cite[Theorem 24.15]{MTe}, and hence quasisimple by
\cite[Corollary 24.17]{MTe}, in fact simple as $Z((L')^F)=1$. In addition, $|L^F/(L')^F|$ is odd (as the unipotent elements are 2-elements). 

Let $g$ be a unipotent element in $P$. Then $gU \in (P/U)'$ by the above paragraph, since the order of $gU$ in  $P/U$ is a power of $2$. Since the unipotent elements of $(P/U)'$ are real, there exist $u_1,u_2\in L'U$ conjugate to $g$ in $L'U$ such that $u_1u_2\in U$. The structure of $U$ is known, in particular, the consecutive factors of the lower central series of $U$ are of exponent 2 \cite[Proposition  17.3]{MTe}.  
If  $G=E_7(q)$, then $U$ is abelian of exponent 2 \cite[Table 5]{Ko}, and hence $e(g)\leq 4$. If  $G=E_8(q)$, then $U$ is nilpotent of class 2 (see \cite[Table 5]{Ko}), and hence $U$ is of exponent at most $4$ (in fact, equal to $4$). Since $P$ contains a Sylow $2$-subgroup of $G$, we conclude that $e(g)\leq 8$ if $g\in G$ is unipotent.

Let $g=su=us\in G$ be the Jordan decomposition of $g$. 
If $s=1$ then $g$ is unipotent, so we are done. Suppose that $s \neq 1$. If $G=E_7(q)$, then every semisimple component of $C_{\mathbf{G}}(s)$ is either classical or of type $E_6$ and $u$ is real in $C_G(s)$, see the proof of Theorem 10.13 in \cite{TZ04}; as $s$ is real in $G$, we conclude that $e(g)\leq 4$. Assume now that $G=E_8(q)$ with $q$ even. Then every semisimple component of $C_G(s)$ is either classical or of type $E_6,E_7$. This follows from   \cite[Table 5.1]{LSS} as $C_G(s)$  contains a maximal torus of $G$. Since $u\in C_G(s)$, by the above the conjugacy class exponent of $u$ in  $ C_G(s)$
 is at most 4, and $s$ is real. Hence $e(g)\leq 8$. 

 Let $G=E_8(q)$ with $5$ dividing $q$, and let $g \in G$. Let $g=su$ be the Jordan decomposition of $g$.  

Suppose first that  $s=1$ so $g=u$ is unipotent. Then $u^{25}=1$ by \cite{Lw}. 
By \cite[the proof of Lemma 10.12]{TZ04}, $u$ is $G$-conjugate to $u^l$ if $l\equiv 1\pmod 5$. In particular, the elements  $u,u^6$ are conjugate. As $u  (u^6)^4 = u^{25} = 1$, we have $e(u)\leq 5$.
 
If $u=1$ then $g=s$ is real by \cite[Proposition 3.11(ii)]{TZ05}, hence $e(g)=2$. So we assume $s\neq 1$, $u\neq 1$, in particular, $s$ is not regular. 

Let $\mathbf{G}$ be the simple algebraic group of type $ E_8$, and let
$\sigma$ be the Frobenius endomorphism of $\mathbf{G}$ such that $G=C_{\mathbf{G}}(\sigma)$. Note that this group is usually denoted by $\mathbf{G}^\sigma$ in the literature, and we use this notation below whenever it is convenient. Obviously, we have $C_G(s)=C_{\mathbf{G}}(s)^\sigma$. We can apply Lemma \ref{u26} 
as soon as we observe that 
$C_{\mathbf{G}}(s)^\sigma$
has no composition factor of type $E_8$ and $C_G(s)$ has no Suzuki group composition factor.  The former is true by Lemma \ref{e88}. 

Observe that the non-abelian composition factors of $C_G(s)$ are simple groups of Lie type in defining characteristic $5$, none of which is a Suzuki group ${}^2B_2(2^{2k+1})$, 
the abelian composition factors are $5'$-groups. This facts follow for instance
from  \cite[Lemma 2.2]{TZ04}. (In the  degenerate case with 
  $k=0$ the group  ${}^2B_2(2)$ is not simple  \cite[Table 24.1, p. 208]{MTe}.) 
Now, by  Lemma \ref{u26}, we have $e_{C_G(s)}(u)\leq 3$ by Corollary \ref{cor881}.
Therefore,  the product of at most $3$ conjugates of $u$ in $C_G(s)$ lies in $\langle s\rangle$. As $s$ is real \cite[Proposition  3.1(ii)]{TZ05}, it follows that $e(g)\leq 6$.    
\end{proof}

\section{Proofs of Theorems \ref{tt2}, \ref{tt3}, \ref{ag9}, \ref{u55}, \ref{s67} and Proposition \ref{uuu3}.} 

\begin{proof}[Proof of Theorem {\rm \ref{tt2}}]
Item $(1)$ follows from Lemma \ref{ma3}, item $(2)$ follows from Theorem \ref{sporadicthm}, item $(3)$ follows from \cite[Theorem 4.2, p.240]{AH}, item $(4)$ follows from Proposition \ref{s4q}, item $(5)$ follows from Lemmas \ref{g24} and \ref{ma3}, 
item $(6)$ follows from Theorem\ref{sf4} and Lemma \ref{2sf4},
item $(7)$ follows from Lemma \ref{ma3}, item $(8)$ follows from Lemmas \ref{ma3} and \ref{u3q1}, item $(9)$ follows from Lemmas \ref{6ps}, \ref{s25} and \ref{s53}.\end{proof}

 \begin{proof}[Proof of Theorem {\rm \ref{tt3}}]
By Table \ref{Table_classical}, $e(g)$ does not exceed the value indicated in the theorem. The opposite inequality is proven for items  (1),(2) in Corollary \ref{boundsPSL}. Items (3),(4) follow from Corollary \ref{oo7}.\end{proof}

\begin{proof}[Proof of Theorem {\rm \ref{ag9}}] See  the references in Tables 1,2.
\end{proof}

\begin{proof}[Proof of Theorem {\rm \ref{u55}}]
Item (1) follows from Lemma \ref{u5n}, Corollary \ref{cor881} and Theorem \ref{ue67}. 
Items (2),(3) follow from \cite[Theorem 4.1 in Chapter 4]{AH} and Theorems \ref{sf4} and \ref{ue67}. Item (5) is exactly \cite[Theorem 1.4]{TZ05}. Recall that a good prime is a prime that is not bad; there are no bad primes for the root system $A_r$; the only bad prime for $B_r$, $C_r$, $D_r$ is $2$; the bad primes for $E_6$, $E_7$, $F_4$, $G_2$ are $2,3$; the bad primes for $E_8$ are $2,3,5$; see \cite[Page 178]{Sbo}.
\end{proof}

\begin{proof}[Proof of Theorem {\rm \ref{s67}}] 
For item (1) see Table \ref{Table_classical} and  Theorem \ref{gg6}, Lemmas \ref{33s} and \ref{44s} for $G=\PSL_n(q)$ and $q=2,3,4$, respectively. Item (2) follows from Lemmas \ref{u5n}, \ref{5eg}, \ref{un2} and \ref{u42}. Items (3),(4) follow from Theorems \ref{orth4} and \ref{666}, respectively.
\end{proof}

\begin{proof}[Proof of Proposition {\rm \ref{uuu3}}]
Item $(1)$ follows from Theorem \ref{gg6}, item $(2)$ follows from Lemma \ref{33s}, item $(3)$ follows from Lemma \ref{pu3}.
\end{proof}

Acknowledgements. We are grateful to Frank L\"ubeck for sending to us his computation of minimal order semisimple elements in finite simple groups of types $E_6(q)$ and ${}^2E_6(q)$.

\end{document}